\documentclass{amsart}
\usepackage{tikz}

\usepackage[a4paper,margin=1.25in]{geometry}
\usetikzlibrary{cd}
\usepackage[utf8]{inputenc}
\usepackage[
    bookmarks=true,
    bookmarksnumbered=true,
    bookmarkstype=toc,hidelinks]{hyperref}

\usepackage{amsthm}
\usepackage{amsmath}
\usepackage{amssymb}
\usepackage{amsfonts}
\usepackage{mleftright}
\usepackage{mathrsfs}
\usepackage{physics}
\DeclareMathOperator{\Aut}{Aut}
\DeclareMathOperator{\Gal}{Gal}

\DeclareMathOperator{\Ker}{Ker}

\DeclareMathOperator{\Inn}{Inn}

\DeclareMathOperator{\End}{End}
\DeclareMathOperator{\Mod}{Mod}

\DeclareMathOperator{\id}{id}
\DeclareMathOperator{\GL}{GL}

\DeclareMathOperator{\ab}{ab}

\DeclareMathOperator{\Out}{Out}

\DeclareMathOperator{\rec}{rec}

\DeclareMathOperator{\pf}{pf}

\theoremstyle{definition}
\newtheorem{dfn}{Definition}[section]
\newtheorem*{dfn*}{dfn}
\newtheorem{rem}[dfn]{Remark}
\theoremstyle{plain}
\newtheorem{thm}[dfn]{Theorem}

\newtheorem{thma}{Theorem}

\newtheorem*{thm*}{Theorem}
\newtheorem{prop}[dfn]{Proposition}
\newtheorem*{prop*}{Proposition}

\newtheorem{lem}[dfn]{Lemma}
\newtheorem*{lem*}{Lemma}

\newtheorem*{claim*}{Claim}
\newcommand{\Q}{\mathbb{Q}}
\newcommand{\Z}{\mathbb{Z}}

\title[Anabelian aspects of the outer automorphism groups]{Anabelian aspects of the outer automorphism groups of the absolute Galois groups of mixed-characteristic local fields}
\author{Kaiji Kondo}
\date{}


\address{Research Institute for Mathematical Sciences, Kyoto University, Kyoto 606-8502, JAPAN}
\email{kkondo@kurims.kyoto-u.ac.jp}

\begin{document}

\begin{abstract}
    In the present paper, we study the outer automorphism groups of the absolute Galois groups of mixed-characteristic local fields from the point of view of anabelian geometry.
    In particular, we show that, under certain mild assumptions, the image of the natural homomorphism from the automorphism group of a mixed-characteristic local field to the outer automorphism group of the associated absolute Galois group is not a normal subgroup.
    Furthermore, we show that, for the absolute Galois group of a mixed-characteristic local field satisfying certain assumptions, there exist a continuous representation and a continuous automorphism of the group such that the former is irreducible, abelian, and crystalline, but the continuous representation obtained as the composite of the former with the latter is not even Hodge-Tate.
    These results significantly generalize previous works by Hoshi and Nishio.
    A key observation in obtaining these results is to focus on the analogy between the mapping class groups of topological surfaces and the outer automorphism groups of the absolute Galois groups of mixed-characteristic local fields.
    To the best of the author's knowledge, this is the first work applying results from the theory of mapping class groups to the anabelian geometry of  mixed-characteristic local fields, going beyond a mere analogy between the two.
\end{abstract}

\maketitle

\tableofcontents

\section*{Introduction}

Let $p$ be a prime number, $k$ a finite extension of $\Q_{p}$, and $\overline{k}$ an algebraic closure of $k$.
We write $G_{k}\stackrel{\mathrm{def}}{=} \Gal(\overline{k}/k)$ for the absolute Galois group of $k$ determined by the algebraic closure $\overline{k}$ and $\Out(G_{k})$ for the group of outer automorphisms of the group $G_{k}$ [or, equivalently, the group
 of outer continuous automorphisms of the profinite group $G_{k}$ --- cf.\ \cite{3}, Proposition 3.3, $\rm(\hspace{.18em}i\hspace{.18em})$, $\rm(\hspace{.08em}ii\hspace{.08em})$].
In the present paper, we study the outer automorphism group $\Out(G_{k})$ from the point of view of anabelian geometry.

A natural and fundamental question is to what extent various kinds of data concerning a field $k$ can be recovered from the group $G_{k}$.
In particular, traditional approaches in anabelian geometry have focused on the question of whether the functor ``taking absolute Galois groups'' is fully faithful.
Unfortunately, this functor from the groupoid of fields isomorphic to finite extensions of $\Q_p$ equipped with algebraic closures to the groupoid of profinite groups is faithful but not full if $p$ is odd [cf., e.g., the discussion given at the final portion of \cite{6}, chapter Chapter V\hspace{-1.2pt}I\hspace{-1.2pt}I, \S5].
For example, if $p$ is an odd prime, and $k/\Q_{p}$ is Galois but not abelian, then there exists a pair $(k^{\prime}, \overline{k^{\prime}})$, where $k^{\prime}$ is a finite extension of $\Q_p$, and $\overline{k^{\prime}}$ is the algebraic closure of $k^{\prime}$, such that $k$ and $k^{\prime}$ are not isomorphic as fields, but the absolute Galois groups $G_k$ and $G_{k^{\prime}} \stackrel{\mathrm{def}}{=} \Gal(\overline{k^{\prime}}/k^{\prime})$ are isomorphic as [topological] groups [cf. \cite{2}, Theorem F, $\rm(\hspace{.18em}i\hspace{.18em})$].

Write $\Aut(k)$ for the group of automorphisms of the field $k$.
Then we have a natural injective homomorphism $\Aut(k) \hookrightarrow \Out(G_{k})$ of groups [cf., e.g., \cite{3}, Proposition 2.1].
In the present paper, let us regard $\Aut(k)$ as a [necessarily finite] subgroup of $\Out(G_{k})$ by means of this injective homomorphism:
\begin{align*}
    \Aut(k) \subset \Out(G_{k}).
\end{align*}
We refer to the subgroup as a \textit{field-theoretic subgroup}.
Here, we note that it is well-known [cf., e.g., the discussion given at the final portion of \cite{6}, Chapter V\hspace{-1.2pt}I\hspace{-1.2pt}I, \S5] that [although a similar equality always holds for a finite extension of $\Q$ by the Neukirch-Uchida theorem --- cf., e.g., \cite{6}, Chapter X\hspace{-1.2pt}I\hspace{-1.2pt}I, Corollary 12.2.2], in general, the equality $\Aut(k)=\Out(G_{k})$ does not hold.
Thus, we are interested in the difference between $\Aut(k)$ and $\Out(G_{k})$.

Let us recall that various ``ring-theoretic characterizations'' of the subgroup $\Aut(k) \subset \Out(G_{k})$ have been studied in works such as \cite{13}, \cite{7}, and \cite{4}.
\begin{thma}[\cite{13}, \cite{7}, \cite{4}]\label{characterization of Aut k}
   Let $\alpha$ be an automorphism of the topological group $G_{k}$.
   Then the following conditions are equivalent:
   \begin{enumerate}
       \item [(1)] The image of $\alpha$ by the natural homomorphism $\Aut(G_{k}) \to \Out(G_{k})$ lies in the subgroup $\Aut(k) \subset \Out(G_{k})$.
       \item [(2)] For every finite dimensional continuous representation $\rho \colon G_{k} \to \GL_{n}(\Q_{p})$ of $G_{k}$ that is Hodge-Tate, the composite $G_{k} \stackrel{\alpha}{\to} G_{k} \stackrel{\rho}{\to} \GL_{n}(\Q_{p})$ is Hodge-Tate.
       \item [(3)] The automorphism $\alpha$ is compatible with the filtration on $G_{k}$ given by the [positively indexed] higher ramification subgroups in the upper numbering.
       \item [(4)] There exists an automorphism of topological groups [but not necessarily of topological fields] $(\widehat{\overline{k}})_{+} \stackrel{\sim}{\longrightarrow} (\widehat{\overline{k}})_{+}$ --- where we write $\widehat{\overline{k}}$ for the $p$-adic completion of $\overline{k}$ and $(\widehat{\overline{k}})_{+}$ for the topological group obtained by forming the underlying additive module of $\widehat{\overline{k}}$ --- that is compatible with the natural action $G_{k} \curvearrowright \widehat{\overline{k}}$ [with respect to $\alpha$].
   \end{enumerate}
\end{thma}
As Theorem A shows, the elements of $\Aut(k)$ are characterized as outer automorphisms that preserve various ``ring-theoretic structures and properties'' such as the ramification filtration, the Hodge–Tate-ness of continuous $p$-adic representations, and the continuous $G_{k}$-module $(\widehat{\overline{k}})_{+}$.
However, Theorem A does not reveal the difference between $\Out(G_{k})$ and $\Aut(k)$.

As a certain refinement of some observations concerning the implication $(2) \Rightarrow (1)$ of Theorem A, the following theorem was proved in \cite{11}.

\begin{thma}[\cite{11}]
    Suppose that $p$ is an odd prime number, and that $k/\Q_{p}$ is an abelian extension of even degree $d$.
     Then there exist a $\Q_{p}$-vector space $V$ of dimension $d$, a continuous representation $\rho \colon G_{k}\to \Aut_{\Q_{p}}(V)$ that is irreducible, abelian, and crystalline [hence also Hodge-Tate], and a group automorphism $\varphi \colon G_{k} \stackrel{\sim}{\longrightarrow} G_{k}$ such that
     $\rho \circ \varphi$ is not Hodge-Tate.
\end{thma}
In the present paper, we generalize this result to the case where the Galois extension $k/\Q_{p}$ is not necessarily abelian.
That is to say, we obtain the following theorem.
\begin{thma}
   Suppose that $p$ is an odd prime number, and that $k/\Q_{p}$ is a finite Galois extension of even degree $d$.
     Then there exist a $\Q_{p}$-vector space $V$ of dimension $d$, a continuous representation $\rho \colon G_{k}\to \Aut_{\Q_{p}}(V)$ that is irreducible, abelian, and crystalline [hence also Hodge-Tate], and a group automorphism $\varphi \colon G_{k} \stackrel{\sim}{\longrightarrow} G_{k}$ such that
     $\rho \circ \varphi$ is not Hodge-Tate.
\end{thma}
At the time of writing, it is not clear to the author whether or not a similar assertion to Theorem C in the case where $k/\Q_{p}$ is a Galois extension of odd degree holds.
We will discuss the reason why the proof of Theorem C does not work in the case where the extension $k/\Q_{p}$ is a Galois extension of odd degree in Remark \ref{odd case}.

Here, we review an outline of the proof of Theorem B and recall the technical difficulties to generalize the argument to the situation of Theorem C.

To prove Theorem B, it is important to show that there exists an automorphism $\varphi$ of $G_k$ such that the induced automorphism $\varphi^{\times}$ of $k^{\times}$, which is obtained via the mono-anabelian reconstruction algorithm, does not preserve the subgroup $\Z_p^{\times} \subset k^{\times}$ of $k^{\times}$.
When $k/\Q_p$ is a quadratic extension, we construct such an automorphism of $G_k$ explicitly by applying the techniques developed in \cite{1} [cf.\ Theorem D below].

When $k/\Q_p$ is a finite abelian extension of even degree, the desired existence reduces to the case where $k/\Q_p$ is a quadratic extension in the following way.
If $ k/\Q_p$ is a finite abelian extension of even degree, then there exists a quadratic intermediate field $k^{\prime}$ of $k/\Q_p$ such that
\begin{align*}
    G_k \subset G_{k^{\prime}} \stackrel{\mathrm{def}}{=} \Gal(\overline{k}/k^{\prime})
\end{align*}
is a characteristic subgroup [cf.\ \cite{1}, Lemma 2.6, $\rm(\hspace{.18em}i\hspace{.18em})$].
Thus, we obtain an automorphism $\varphi^{\prime}$ of $G_{k^{\prime}}$ such that $(\varphi^{\prime})^{\times}$ does not preserve the subgroup $\Z_{p}^{\times} \subset (k^{\prime})^{\times}$ of $(k^{\prime})^{\times}$.
Moreover, the restriction $\varphi$ to $G_{k}$ of the automorphism $\varphi^{\prime}$ satisfies the desired property.

However, the above problem can no longer be reduced to the quadratic case in the case where $ k/\Q_p$ is a Galois extension of even degree that is not abelian.
Moreover, the action of $\Aut(G_k)$ on the $\Q_{p}$-vector space $k_+$ obtained by forming the underlying additive module of the field $k$ induced by the mono-anabelian reconstruction algorithm may be complicated.
A large portion of \S1 of the present paper is devoted to overcoming this difficulty.
\\\par
Let us make some preparations to state a key result for the generalization.
The following theorem is fundamental in the study of the outer automorphism groups of the absolute Galois groups of mixed-characteristic local fields.
\begin{thma}[\cite{6}]
    Suppose that $p$ is an odd prime number, and that $k/\Q_{p}$ is a finite extension of degree $d$ and of residue degree $f$.
   Then there exist $\sigma$, $\tau$, $x_{0}, \ldots, x_{d}\in G_{k}$, positive integers $s$, $t$, an element $x_{0}^{\prime} \in \overline{\langle\tau, x_{0} \rangle}$, and an element $x_{1}^{\prime} \in \overline{\langle \sigma,\tau,x_{1}\rangle}$, where ``$\overline{\langle S \rangle}$'' denotes the closed subgroup of $G_{k}$ topologically generated by $S$, such that the following conditions hold:
    \begin{enumerate}
    \item [(1)] $G_{k}$ is presented as the profinite group topologically generated by $\sigma$, $\tau$, $x_{0}, \ldots, x_{d}\in G_{k}$ and subject to the relations described in the conditions (3) and (4) below.
  \item [(2)] The wild inertia subgroup $P_{k}$ of $G_{k}$ is topologically normally generated by $x_{0},\ldots,x_{d}$.
  \item [(3)] The elements $\sigma$, $\tau$ satisfy the relation $\sigma \tau \sigma^{-1}=\tau^{p^{f}}$.
  \item [(4)] In addition, the generators satisfy one further relation: 
   \begin{enumerate}
    \item for even $d$,
    \begin{align*}
        \sigma x_{0} \sigma^{-1}=(x_{0}^{\prime})^{t} x_{1}^{p^s}[x_{1},x_{2}][x_{3},x_{4}]\cdots[x_{d-1},x_{d}];
    \end{align*}
    \item for odd $d$, 
    \begin{align*}
         \sigma x_{0} \sigma^{-1}=(x_{0}^{\prime})^{t} x_{1}^{p^s}[x_{1},x_{1}^{\prime}][x_{2},x_{3}]\cdots[x_{d-1},x_{d}].
    \end{align*}
  \end{enumerate}
 
\end{enumerate}
    \end{thma}
We write $\widehat{k^{\times}}$ for the profinite completion of the multiplicative group of $k$ and $k_{+}$ for the $\Q_{p}$-vector space obtained by forming the underlying additive module of $k$.
When $p$ is odd, we fix topological generators $\sigma$, $\tau$, $x_{0}, \ldots, x_{d}$ of $G_{k}$ as in Theorem D.
Let $F \colon G_{k} \to k_{+}$ be the homomorphism obtained as the composite of the natural surjection $G_{k} \twoheadrightarrow G_{k}^{\ab}$, the inverse of the isomorphism induced by the local reciprocity map $\widehat{k^{\times}} \stackrel{\sim}{\longrightarrow} G_{k}^{\ab}$, the projection $\widehat{k^{\times}} \twoheadrightarrow \mathcal{O}_{k}^{\times}$ determined by a choice of a uniformizer of $\mathcal{O}_{k}$, and the $p$-adic logarithm $\log_{k} \colon \mathcal{O}_{k}^{\times} \to k_{+}$.
Moreover, if $p$ is odd, then, for each $i \in \{1, 2, \ldots, d\}$, we write $y_i \in k_{+}$ for the image of $x_i$ by the map $F$.
Note that, when $d>1$ and $p$ is odd, the $d$ elements $y_{1}, \ldots, y_{d}$ form a basis of the $\Q_{p}$-vector space $k_{+}$ [cf.\ \cite{1}, Lemma 1.2; \cite{1}, Lemma 1.3; \cite{3}, Proposition 3.11, $\rm(i\hspace{-.08em}v\hspace{-.06em})$].

We obtain the following result by considering the mono-anabelian reconstruction algorithm, together with the ``Dehn twists'' in $\Out(G_{k})$ [cf., e.g., \cite{14}, \cite{20}].
This result is interesting in its own right and, moreover, plays an important role in the proof of Theorem C.
\begin{thma}
    Suppose that the finite extension $k/\Q_{p}$ is of degree $d>1$, and that $p$ is odd.
    We write $\Tr_{k/\Q_{p}} \colon k_{+} \to (\Q_{p})_{+}$ for the trace map of the extension $k/\Q_{p}$.
    Then we obtain the following equality of sub-$\Q_{p}$-vector spaces of $k_{+}$:
    \begin{align*}
    \Ker(\Tr_{k/\Q_{p}}) = \left\{
    \begin{aligned}
    &\langle y_{1},y_{3},\ldots,y_{d} \rangle_{\Q_{p}}\ (d : even)\\
    &\langle y_{2},y_{3},\ldots,y_{d} \rangle_{\Q_{p}}\  (d : odd),
    \end{aligned}
    \right.\ 
    \end{align*}
    where ``$\langle S \rangle_{\Q_{p}}$'' denotes the sub-$\Q_{p}$-vector space of $k_{+}$ generated by $S$.
\end{thma}
 We obtain a useful description of the sub-$\Q_{p}$-vector space $(\Q_{p})_{+} \subset k_{+}$ in terms of the basis $y_{1}, \ldots, y_{d}$ by applying Theorem E.
This description enables us to apply the techniques developed in \cite{11} to prove Theorem C, which is the first main result of the present paper.

Here, we note that the generators and relations of $G_{k}$ in Theorem D are non-canonical, and that the proof of Theorem D is purely group-theoretic. 
Moreover, it is nontrivial to relate these generators and relations of $G_{k}$ to the ring structure of $k$. 
From this perspective, it is of independent interest the issue of whether or not, under rather mild assumptions such as those in Theorem E, one can obtain a result that directly connects the generators and relations of $G_{k}$ in Theorem D with the ring structure of $k$.
\\\par
Next, let us discuss the second main theorem of the present paper.
First, we recall that one interesting problem  from the perspective of algorithmic anabelian geometry is the issue of whether or not the subgroup $\Aut(k) \subset \Out(G_{k})$ can be recovered functorially and group-theoretically from $G_{k}$.  
For example, as shown in \cite{2}, Theorem 6.12, under suitable assumptions on $k$, the set of $\Out(G_{k})$-conjugates of the subgroup $\Aut(k) \subset \Out(G_{k})$ can be recovered functorially and group-theoretically.

\begin{thma}[\cite{2}]
Suppose that $k$ is Galois-specifiable [cf.\ \cite{2}, Definition 6.1].  
Then the set of $\Out(G_{k})$-conjugates of the subgroup $\Aut(k) \subset \Out(G_{k})$ is equal to the set of strictly quasi-geometric subgroups of $\Out(G_{k})$ [cf.\ \cite{2}, Definition 6.5, (i), (ii)].
In particular, the set of $\Out(G_{k})$-conjugates of the subgroup $\Aut(k) \subset \Out(G_{k})$ can be recovered functorially and group-theoretically.
\end{thma}

It follows formally that if one can establish a functorial group-theoretic algorithm that recovers the subgroup $\Aut(k) \subset \Out(G_{k})$ from $G_{k}$, then it is necessary that the subgroup $\Aut(k) \subset \Out(G_{k})$ is a normal subgroup [cf.\ Remark \ref{summary of main result}].
However, the subgroup $\Aut(k) \subset \Out(G_{k})$ is not necessarily a normal subgroup in general as shown in \cite{2} and \cite{1}.
\begin{thma}[\cite{2}]
    Suppose that $p$ is an odd prime number, and that $k=\Q_{p}(\zeta_{p}, \sqrt[p]{p})$, where $\zeta_{p} \in \overline{k}$ is a primitive $p$-th root of unity, and $\sqrt[p]{p} \in \overline{k}$ is a $p$-th root of $p$. 
    Then the subgroup $\Aut(k) \subset \Out(G_{k})$ is not a normal subgroup.
\end{thma}

\begin{thma}[\cite{1}]
    Suppose that $p$ is an odd prime number, and that $k/\Q_{p}$ is a finite abelian extension of even degree.
    Then the set of $\Out(G_{k})$-conjugates of the subgroup $\Aut(k) \subset \Out(G_{k})$ is infinite.
    In particular, the subgroup $\Aut(k) \subset \Out(G_{k})$ is not normal, and there exist infinitely many distinct [necessarily finite] subgroups of $\Out(G_{k})$ isomorphic to $\Aut(k)$.
\end{thma}
Theorem G and Theorem H were proved by different techniques.
In the present paper, we prove the non-normality of $\Aut(k) \subset \Out(G_{k})$ in a more general form that includes both Theorem G and Theorem H.
The following result is the second main theorem of the present paper.

\begin{thma}
     Suppose that $p$ is an odd prime number, and that $k/\Q_p$ is a finite Galois extension of degree greater than one.
     Then the set of $\Out(G_{k})$-conjugates of the subgroup $\Aut(k) \subset \Out(G_{k})$ is infinite.
    In particular, the subgroup $\Aut(k) \subset \Out(G_{k})$ is not normal, and there exist infinitely many distinct [necessarily finite] subgroups of $\Out(G_{k})$ isomorphic to $\Aut(k)$.
\end{thma}
We apply essentially different arguments to prove Theorem I depending on whether $[k \colon \Q_{p}]$ is even or odd.
Suppose that $[k \colon \Q_{p}]$ is even.
Let $\varphi$ be an automorphism of $G_{k}$ as in Theorem C.
We write $\varphi_{+}$ for the automorphisms, which is induced from $\varphi$ by the mono-anabelian reconstruction algorithm.
Suppose that $k/\Q_{p}$ is a finite Galois extension.
It follows from the discussion in \S1 that the automorphism $\varphi_{+}$ does not preserve the $\Q_{p}$-subspace $(\Q_{p})_{+} \subset k_{+}$ of $k_{+}$.
We prove Theorem I in the case where $k/\Q_{p}$ is a finite Galois extension of even degree by combining this result with the argument of \cite{1}.

On the other hand, when $k/\Q_{p}$ is of odd degree, since it is not clear to the author whether or not there exists an element of $\Out(G_{k})$ that does not preserve the $\Q_{p}$-subspace $(\Q_{p})_{+} \subset k_{+}$, the same technique cannot be applied.
A key observation in the proof of Theorem I in the case where $k/\Q_{p}$ is a finite Galois extension of odd degree greater than one is obtained by focusing on the analogy between $\Out(G_{k})$ and the mapping class group of a surface, as developed in \cite{14}, \cite{20}, and related works.

In the present paper, we prove Theorem I by comparing, from the perspective of the natural $p$-adic Lie group structures, the image of the ``subgroup of $\Out(G_{k})$ corresponding to the mapping class group'' [cf.\ Remark \ref{mapping class group meaning}] by the group homomorphism $\Out(G_{k}) \to \Aut_{\Q_{p}}(k_{+})$ induced by the mono-anabelian reconstruction algorithm of \cite{3}, Proposition 3.11, $\rm(i\hspace{-.08em}v\hspace{-.06em})$,
with the image of $Z_{\Out(G_{k})}(\Aut(k))$ by the same homomorphism.

Here, we note that the case of even degree can also be proved by the same method as in the odd degree case. 
However, as discussed above, we instead apply another method, which has the additional advantage of providing explicit elements in $\Out(G_{k}) \setminus N_{\Out(G_{k})}(\Aut(k))$ [cf. Remark \ref{odd-case}].
\\\par
Although \cite{14} and \cite{20} studied certain analogies between the outer automorphism groups of absolute Galois groups of mixed-characteristic local fields and the mapping class groups of topological surfaces, there have been very few results that apply this analogy to objects reflecting the ring structure of $k$. 
In the present paper, some results from the theory of mapping class groups are applied to prove the non-normality of field-theoretic subgroups and to establish nontrivial phenomena concerning $p$-adic representations. 
To the best of the author's knowledge, this is the first work that applies results from the theory of mapping class groups to the anabelian geometry of mixed-characteristic local fields in a way that goes beyond a mere analogy.

\setcounter{section}{-1}

\section{Notational conventions}

\noindent{\bf Monoids}. In the present paper, every ``monoid'' is assumed to be commutative. Let $M$ be a monoid.
[The monoid operation of $M$ will be written multiplicatively]. We shall write $M^{\pf}$ for the \textit{perfection} of $M$ [i.e., the monoid obtained by
forming the inductive limit of the inductive system of monoids
\begin{align*}
    \cdots \to M \to M \to \cdots 
\end{align*}
given by assigning to each positive integer $n$ a copy of $M$, which we write for $I_{n}$, and to each two positive integers $n$, $m$ such that $n$ divides $m$ the homomorphism $I_{n}=M \to M=I_{m}$ given by multiplication by $m/n$].
\vspace{0.3cm}

\noindent {\bf Groups}.
If $G$ is a group, and $H \subset G$ is a subgroup of $G$, then we shall write $N_{G}(H) \subset G$ for the \textit{normalizer} of $H$ in $G$.
Moreover, we shall write $Z_{G}(H)$ for the \textit{centralizer} of $H$ in $G$.
\vspace{0.3cm}

\noindent {\bf Topological groups}.
If $G$ is a topological group, then we shall write $G^{\ab}$  for the \textit{abelianization} of $G$ [i.e., the quotient of $G$ by the closure of the commutator subgroup of $G$].
\vspace{0.3cm}

\noindent {\bf Rings}. In the present paper, every ``ring'' is assumed to be unital, associative, and commutative. If $R$ is a ring, then we shall write $R_{+}$ for the underlying additive group of $R$ and $R^{\times} \subset R$  for the multiplicative group of units of $R$.
\vspace{0.3cm}

\noindent {\bf Mixed-characteristic local fields}. We shall refer to a field isomorphic to a finite extension of $\Q_{p}$, for some prime number $p$, as an \textit{MLF}. Here, ``MLF'' is to be understood as an abbreviation for ``mixed-characteristic local field''. 
Let $k$ be an MLF and $\overline{k}$ an algebraic closure of $k$.
Then we shall write
\begin{itemize}
    \item $\mathcal{O}_{k}$ for the ring of integers of $k$,
    \item  $\mathfrak{m}_{k} \subset \mathcal{O}_{k}$ for the maximal ideal of $\mathcal{O}_{k}$,
    \item $\underline{k} \stackrel{\mathrm{def}}{=}\mathcal{O}_{k}/\mathfrak{m}_{k}$ for the residue field of $\mathcal{O}_{k}$,
    \item $k^{(d=1)} \subset k$ for the [uniquely determined] minimal MLF contained in $k$,
     \item $p_{k}$ for the residue characteristic of $k$,
     \item $d_{k}\stackrel{\mathrm{def}}{=}[k \colon k^{(d=1)}]$ for the degree of the finite extension $k/k^{(d=1)}$,
     \item $f_{k}\stackrel{\mathrm{def}}{=}[\underline{k} \colon \underline{k}^{(d=1)}]$ for the degree of the finite extension $\underline{k}/\underline{k}^{(d=1)}$ [where we write $\underline{k}^{(d=1)}$ for the residue field of the ring of integers of the MLF $k^{(d=1)}$],
    \item $G_{k} \stackrel{\mathrm{def}}{=} \Gal(\overline{k}/k)$ for the absolute Galois group of $k$ determined by the algebraic closure $\overline{k}$,
    \item $G_{k^{(d=1)}}\stackrel{\mathrm{def}}{=}\Gal(\overline{k}/k^{(d=1)})$ for the absolute Galois group of $k^{(d=1)}$ determined by the algebraic closure $\overline{k}$,
    \item  $I_{k} \subset G_{k}$ for the inertia subgroup of $G_{k}$,
    \item $P_{k} \subset I_{k}$ for the wild inertia subgroup of $G_{k}$, 
    \item $\widehat{k^{\times}}$ for the profinite completion of the multiplicative group $k^{\times}$ of $k$,
    \item  $\rec_{k} \colon \widehat{k^{\times}} \stackrel{\sim}{\longrightarrow} G_{k}^{\ab}$ for the isomorphism induced by the reciprocity homomorphism $k^{\times} \hookrightarrow G_{k}^{\ab}$ in local class field theory,
    \item  $\log_{k} \colon \mathcal{O}_{k}^{\times} \to k_{+}$ for the $p_{k}$-adic logarithm, and
    \item $\mathcal{I}_{k}\stackrel{\mathrm{def}}{=}(2p_{k})^{-1}\log_{k}(\mathcal{O}_{k}^{\times}) \subset k_{+}$ for the log-shell of $k$.
\end{itemize}

\noindent {\bf Groups of MLF-type}.
We shall refer to a group isomorphic to the absolute Galois group of an MLF as a \textit{group of MLF-type}.
Let $G$ be a group of MLF-type.
We shall say that a subgroup of $G$ is \textit{open} if it is of finite index.
It follows from \cite{9}, Theorem 0.1, and an elementary argument from general topology that the open subgroups of $G$ determine the structure of a profinite group of $G$, and that this topology of $G$ is a uniquely determined topology of $G$ that gives rise to a structure of profinite group on $G$.
In the remainder of the present paper, we equip $G$ with this profinite group structure.
In particular, an arbitrary isomorphism between groups of MLF-type as abstract groups is necessarily an isomorphism of profinite groups.

Moreover, let us recall [cf.\ \cite{3}, Definition 3.5, Proposition 3.6, Definition 3.10, Proposition 3.11] that there exist functorial group-theoretic algorithms for constructing, from a group of MLF-type $G$,
\begin{itemize}
    \item a prime number $p(G)$,
    \item positive integers $d(G)$, $f(G)$,
    \item subgroups $P(G) \subset I(G) \subset G$ of $G$,
    \item subgroups $\mathcal{O}^{\times}(G) \subset k^{\times}(G) \subset G^{\ab}$, where we write $\rec_{G}$ for the final inclusion $k^{\times}(G) \subset G^{\ab}$, and
    \item a $\Q_{p(G)}$-vector space $k_{+}(G)$
\end{itemize}
 which ``correspond'' to
\begin{itemize}
    \item the prime number $p_{k}$,
    \item the positive integers $d_{k}$, $f_{k}$,
    \item the subgroups $P_{k} \subset I_{k} \subset G_{k}$ of $G_{k}$,
    \item the subgroups $\mathcal{O}_{k}^{\times} \subset k^{\times} \stackrel{\rec_{k}}{\hookrightarrow} G_{k}^{\ab}$, and
    \item the $\Q_{p_{k}}$-vector space $k_{+}$,
\end{itemize}
respectively.

Moreover,  it follows from \cite{3}, Proposition 3.11, $\rm(\hspace{.18em}i\hspace{.18em})$, $\rm(i\hspace{-.08em}v\hspace{-.06em})$; \cite{1}, Lemma 1.2, that we have natural homomorphisms
\begin{align*}
    \Aut(G) \to \Aut(k^{\times}(G)),\ \ \ \Aut(G) \to \Aut_{\Q_{p(G)}}(k_{+}(G)).
\end{align*}
 Since the automorphism of $G^{\ab}$ induced by an inner automorphism of $G$ is trivial, it follows from the constructions of $k^{\times}(G)$ and $k_{+}(G)$ that the automorphisms of $k^{\times}(G)$ and $k_{+}(G)$ induced by an inner automorphism of $G$ are trivial.
  Thus, the above two homomorphisms determine group homomorphisms
  \begin{align*}
      \Out(G) \to \Aut(k^{\times}(G)),\ \ \ \Out(G) \to \Aut_{\Q_{p(G)}}(k_{+}(G)).
  \end{align*}

Let $\alpha$ be an element of $\Out(G)$. 
We write $\alpha^{\times}$ [respectively, $\alpha_{+}$] for the image of $\alpha$ by the homomorphism $\Out(G) \to \Aut(k^{\times}(G))$ [respectively, $\Out(G) \to \Aut_{\Q_{p(G)}}(k_{+}(G))$].
In the present paper, we call $\alpha^{\times}$ and $\alpha_{+}$ \textit{the automorphisms induced from $\alpha$ by the mono-anabelian reconstruction algorithms}.

Let $k$ be an MLF and $\alpha$ an element of $\Out(G_{k})$.
By abuse of notation, we shall denote by $\alpha_{+} \colon k_{+} \stackrel{\sim}{\longrightarrow} k_{+}$, $\alpha^{\times} \colon k^{\times} \stackrel{\sim}{\longrightarrow} k^{\times}$ the respective images of $\alpha_{+} \colon k_{+}(G_{k}) \stackrel{\sim}{\longrightarrow} k_{+}(G_{k})$, $\alpha^{\times} \colon k^{\times}(G_{k}) \stackrel{\sim}{\longrightarrow} k^{\times}(G_{k})$ by the isomorphisms $\Aut(k_{+}(G_{k})) \stackrel{\sim}{\longrightarrow} \Aut(k_{+})$, $\Aut(k^{\times}(G_{k})) \stackrel{\sim}{\longrightarrow} \Aut(k^{\times})$ induced by the natural isomorphisms $k_{+} \stackrel{\sim}{\longrightarrow} k_{+}(G_{k})$, $k^{\times} \stackrel{\sim}{\longrightarrow} k^{\times}(G_{k})$ [cf.\ \cite{3}, Proposition 3.11, $\rm(\hspace{.18em}i\hspace{.18em})$, $\rm(i\hspace{-.08em}v\hspace{-.06em})$].

\newpage

\section{A group-theoretic description of the kernels of absolute trace maps and the application to the study of Aut-intrinsic Hodge-Tate representations}
In the present \S1, let $k$ be an MLF, $\overline{k}$ an algebraic closure of $k$, and $G$ a group of MLF-type.
In the present \S1, we show that the kernel of the trace map for the extension $k/k^{(d=1)}$ can be described ``group-theoretically''.
 Moreover, by means of this group-theoretic description of the kernel of the trace map for $k/k^{(d=1)}$, together with classical techniques from $p$-adic representation theory, we show that if $p_{k}$ is odd, $d_{k}$ is even, and $k$ is an absolutely Galois MLF, i.e., if $k/k^{(d=1)}$ is a Galois extension, then there exist a $\Q_{p_{k}}$-vector space $V$ of dimension $d_{k}$, a continuous representation $\rho \colon G_{k}\to \Aut_{\Q_{p_{k}}}(V)$ that is irreducible, abelian, and crystalline [hence also Hodge-Tate], and a group automorphism $\varphi \colon G_{k} \stackrel{\sim}{\longrightarrow} G_{k}$ such that
     $\rho \circ \varphi$ is not Hodge-Tate.
This result generalizes the main theorem of \cite{11}.

\begin{thm}[Jannsen-Wingberg]\label{generator and relation of group of MLF-type} 
   Suppose that $p(G)$ is an odd prime number.
   Then there exist $\sigma$, $\tau$, $x_{0}, \ldots, x_{d(G)}\in G$, positive integers $s$, $t$, an element $x_{0}^{\prime} \in \overline{\langle\tau, x_{0} \rangle}$; and an element $x_{1}^{\prime} \in \overline{\langle \sigma,\tau,x_{1}\rangle}$, where ``$\overline{\langle S \rangle}$'' denotes the closed subgroup of $G$ topologically generated by $S$, such that the following conditions hold:
    \begin{enumerate}
     \item [(1)] $G$ is presented as the profinite group topologically generated by $\sigma$, $\tau$, $x_{0}, \ldots, x_{d(G)}\in G$ and subject to the relations described in the conditions (3) and (4) below.
  \item [(2)] The closed normal subgroup $P(G)$ of $G$ is topologically normally generated by $x_{0},\ldots,x_{d(G)}$.
  \item [(3)] The elements $\sigma$, $\tau$ satisfy the relation $\sigma \tau \sigma^{-1}=\tau^{p(G)^{f(G)}}$.
  \item [(4)] In addition, the generators satisfy one further relation: 
   \begin{enumerate}
    \item for even $d(G)$,
    \begin{align*}
        \sigma x_{0} \sigma^{-1}=(x_{0}^{\prime})^{t} x_{1}^{p(G)^s}[x_{1},x_{2}][x_{3},x_{4}]\cdots[x_{d(G)-1},x_{d(G)}];
    \end{align*}
    \item for odd $d(G)$, 
    \begin{align*}
         \sigma x_{0} \sigma^{-1}=(x_{0}^{\prime})^{t} x_{1}^{p(G)^s}[x_{1},x_{1}^{\prime}][x_{2},x_{3}]\cdots[x_{d(G)-1},x_{d(G)}].
    \end{align*}
  \end{enumerate}
 
\end{enumerate}
    
\end{thm}

\begin{proof}
    This assertion follows from \cite{6}, Theorem 7.5.14, together with \cite{3}, Proposition 3.6.
\end{proof}

In the remainder of the present paper, we apply the notational conventions introduced in the statement of Theorem \ref{generator and relation of group of MLF-type} in each of the situations in which $p(G)$ is assumed to be odd.
Moreover, for each $i=1,2,\ldots,d(G)$, write $y_{i} \in k_{+}(G)$ for the image of $x_{i}$ in $k_{+}(G)$ by the composite $P(G) \hookrightarrow I(G) \to \mathcal{O}^{\times}(G) \to k_{+}(G)$ [cf.\ conditions (1), (2) of Theorem \ref{generator and relation of group of MLF-type}; \cite{3}, Definition 3.10, $\rm(\hspace{.18em}i\hspace{.18em})$, $\rm(\hspace{.08em}ii\hspace{.08em})$, $\rm(\hspace{.06em}v\hspace{.06em})$].

We recall that the topological group $k_{+}(G)$ has a natural structure of $\Q_{p(G)}$-vector space of dimension $d(G)$ [cf.\ \cite{1}, Lemma 1.2].
Moreover, one verifies easily that the isomorphism of topological groups $k_{+}(G_{k}) \stackrel{\sim}{\longrightarrow} k_{+}$ of \ \cite{3}, Proposition 3.11, $\rm(i\hspace{-.08em}v\hspace{-.06em})$ is also an isomorphism of $\Q_{p_{k}}$-vector spaces [here, we have $p_{k} = p(G_{k})$ --- cf.\ \cite{3}, Proposition 3.6].
When $d(G)>1$ and $p(G)$ is odd, the $d(G)$ elements $y_{1},\ldots,y_{d(G)}$ defined in the preceding paragraph form a basis of the $\Q_{p(G)}$-vector space $k_{+}(G)$ [cf.\ \cite{1}, Lemma 1.3].

\begin{lem}\label{preservation Ker(Tr)}
    Let $\alpha$ be an element of $\Out(G_k)$, and let $\alpha_{+}$ be the automorphism of the $\Q_{p_k}$-vector space $k_{+}$ induced from $\alpha$ by the mono-anabelian reconstruction algorithm.  
Then we have $\alpha_{+}(\Ker(\Tr_{k/k^{(d=1)}})) = \Ker(\Tr_{k/k^{(d=1)}})$.
\end{lem}

\begin{proof}
    This assertion follows immediately from \cite{1}, Lemma 2.3, $\rm(\hspace{.08em}ii\hspace{.08em})$.
\end{proof}

With the above preparations in place, we can now prove the following assertion, which is the most technically substantial result of the present paper:
\begin{thm}\label{determination of Ker(Tr)}
    Suppose that $p_k$ is odd and that $d_k \geq 2$.
    Then the following equality of sub-$\Q_{p_{k}}$-vector spaces of $k_{+}$ holds:
    \begin{align*}
    \Ker(\Tr_{k/k^{(d=1)}}) = \left\{
\begin{aligned}
&\langle y_{1},y_{3},\ldots,y_{d_{k}} \rangle_{\Q_{p_{k}}}\ (\text{if $d_{k}$ is even})\\
&\langle y_{2},y_{3},\ldots,y_{d_{k}} \rangle_{\Q_{p_{k}}}\  (\text{if $d_{k}$ is odd}),
\end{aligned}
\right.\ 
\end{align*}
 where ``$\langle S \rangle_{\Q_{p_{k}}}$'' denotes the sub-$\Q_{p_{k}}$-vector space of $k_{+}$ generated by ``$S$''.
Here, by abuse of notation, for each integer $i$ satisfying $1 \leq i \leq d_{k}$, we write $y_{i} \in k_{+}$ for the image of $y_{i} \in k_{+}(G_{k})$ by the isomorphism $k_{+}(G_{k}) \stackrel{\sim}{\longrightarrow} k_{+}$ of \cite{3}, Proposition 3.11, $\rm(i\hspace{-.08em}v\hspace{-.06em})$.
\end{thm}

\begin{proof}
    Since $k/k^{(d=1)}$ is a finite separable extension, the map $\Tr_{k/k^{(d=1)}} \colon k_{+} \to k^{(d=1)}_{+}$ is a surjective $\Q_{p_{k}}$-linear map. Consequently, we have the following equality:
    \begin{align*}
        \dim\Ker(\Tr_{k/k^{(d=1)}}) = d_{k} - 1.
    \end{align*}
    
    First, we prove Theorem \ref{determination of Ker(Tr)} in the case where $d_{k} = 2$.
    In this case, we have $\dim \Ker(\Tr_{k/k^{(d=1)}})=1$.
    Fix a $\Q_{p_{k}}$-basis $\alpha_1y_1+\alpha_2y_2$ of $\Ker(\Tr_{k/k^{(d=1)}})$, where $\alpha_{1}, \alpha_{2} \in \Q_{p_{k}}$.
    If $\alpha_2 = 0$, then since $\alpha_1 y_1 \neq 0$, it follows that $\alpha_1 \neq 0$.
    Thus, we conclude that $y_1 \in \Ker(\Tr_{k/k^{(d=1)}})$.
    Next, suppose that $\alpha_{2} \neq 0$.
    Observe that it follows from Theorem \ref{generator and relation of group of MLF-type} that there exists an element $\varphi \in \Aut(G_k)$ satisfying the following equalities:
    \begin{align*}
        \varphi(\sigma)=\sigma,\ \varphi(\tau)=\tau,\ \varphi(x_{0})=x_{0},\ \varphi(x_{1})=x_{1},\ \varphi(x_{2})=x_{2}x_{1}.
    \end{align*}
    It follows from the construction of the natural homomorphism $\Out(G_{k}) \to \Aut_{\Q_{p_{k}}}(k_{+})$ [cf.\ \S 0, \textrm{Groups of MLF-type}] that we obtain $\varphi_{+}(\alpha_{1}y_{1} + \alpha_{2}y_{2}) = (\alpha_{1} + \alpha_{2})y_{1} + \alpha_{2}y_{2}$.
    Moreover, it follows from Lemma \ref{preservation Ker(Tr)} that $\varphi_{+}(\alpha_{1}y_{1} + \alpha_{2}y_{2}) \in \Ker(\Tr_{k/k^{(d=1)}})$.
    Thus, we obtain that
    \begin{align*}
        \alpha_{2}y_{1}=\varphi_{+}(\alpha_{1}y_{1} + \alpha_{2}y_{2})-(\alpha_{1}y_{1} + \alpha_{2}y_{2}) \in \Ker(\Tr_{k/k^{(d=1)}}).
    \end{align*}
    Since $\alpha_{2} \neq 0$, we obtain $y_{1} \in \Ker(\Tr_{k/k^{(d=1)}})$.
    Thus, Theorem \ref{determination of Ker(Tr)} in the case where $d_k = 2$ follows.
    
    Next, we prove Theorem \ref{determination of Ker(Tr)} in the case where $d_{k} \geq 3$.
    Suppose that $d_{k} \geq 3$.
    We define the integer $g$ to be 
   \begin{align*}
   g \stackrel{\mathrm{def}}{=} \left\{
\begin{aligned}
&\frac{d_{k}-2}{2}\ (\text{if $d_{k}$ is even})\\
&\frac{d_{k}-1}{2}\  (\text{if $d_{k}$ is odd}).
\end{aligned}
\right.\ 
\end{align*}
For each integer $i$ satisfying $1 \leq i \leq g$, 
we put
\begin{align*}
    a_{i} \stackrel{\mathrm{def}}{=} \left\{
\begin{aligned}
&x_{2i+1}\ (\text{if $d_{k}$ is even})\\
&x_{2i}\  (\text{if $d_{k}$ is odd}),
\end{aligned}
\right.\ 
b_{i} \stackrel{\mathrm{def}}{=} \left\{
\begin{aligned}
&x_{2i+2}\ (\text{if $d_{k}$ is even})\\
&x_{2i+1}\  (\text{if $d_{k}$ is odd}).
\end{aligned}
\right.\ 
\end{align*}

Thus, it follows from Theorem \ref{generator and relation of group of MLF-type} that, for each integer $i$ satisfying $1 \leq i \leq g$, there exist automorphisms $\varphi_{i}$, $\varphi^{\prime}_{i}$ of $G_k$ satisfying the following equalities:
\begin{align*}
    &\varphi_{i}(\sigma)=\sigma,\ \varphi_{i}(\tau)=\tau,\ \varphi_{i}(a_{j})=a_{j},\ \varphi_{i}(b_{j^{\prime}})=b_{j^{\prime}},\ \varphi_{i}(b_{i}) = b_{i}a_{i},\\
    &\varphi^{\prime}_{i}(\sigma)=\sigma,\ \varphi^{\prime}_{i}(\tau)=\tau,\ \varphi^{\prime}_{i}(a_{j^{\prime}})=a_{j^{\prime}},\ \varphi^{\prime}_{i}(b_{j})=b_{j},\ \varphi^{\prime}_{i}(a_{i}) = a_{i}b_{i}^{-1},\\
    &\varphi_{i}(x_{j})=x_{j},\ \varphi^{\prime}_{i}(x_{j})=x_{j},\ (j=0,1,2\ \text{if $d_{k}$ is even})\\
    &\varphi_{i}(x_{j})=x_{j},\ \varphi^{\prime}_{i}(x_{j})=x_{j},\ (j=0,1\ \text{if $d_{k}$ is odd})
\end{align*}
where $j$ ranges over the integers satisfying $1 \leq j \leq g$, and 
$j^{\prime}$ ranges over the integers satisfying $1 \leq j^{\prime} \leq g$ and $j^{\prime} \neq i$.  
Moreover, when $d_{k} \geq 5$, for each integer $i$ satisfying $1 \leq i\leq g - 1$, there exists an automorphism $\varphi_{i}^{\prime\prime}$ of $G_k$ satisfying the following equalities:
\begin{align*}
   &\varphi_{i}^{\prime\prime}(\sigma)=\sigma,\ \varphi_{i}^{\prime\prime}(\tau)=\tau,\ \varphi_{i}^{\prime\prime}(a_{j})=a_{j},\ \varphi_{i}^{\prime\prime}(b_{j^{\prime}})=b_{j^{\prime}},\ \varphi_{i}^{\prime\prime}(a_{i})=a_{i}b_{i}^{-1}a_{i+1}b_{i+1}a_{i+1}^{-1},\\
   &\varphi_{i}^{\prime\prime}(b_{i})=a_{i+1}b_{i+1}^{-1}a_{i+1}^{-1}b_{i}a_{i+1}b_{i+1}a_{i+1}^{-1},\ \varphi_{i}^{\prime\prime}(a_{i+1})=a_{i+1}b_{i+1}^{-1}a_{i+1}^{-1}b_{i}a_{i+1},\\
   &\varphi^{\prime\prime}_{i}(x_{j})=x_{j},\ (j=0,1,2\ \text{if $d_{k}$ is even})\\
    &\varphi^{\prime\prime}_{i}(x_{j})=x_{j},\ (j=0,1\ \text{if $d_{k}$ is odd})
\end{align*}
where $j$ ranges over the integers satisfying $1 \leq j \leq g$ and $j \neq i,\ i+1$, and 
$j^{\prime}$ ranges over the integers satisfying $1 \leq j^{\prime} \leq g$ and $j^{\prime} \neq i$.
[See Remark \ref{mapping class group meaning} for the motivation of the definitions of $\varphi_{i}$, $\varphi^{\prime}_{i}$, $\varphi^{\prime\prime}_{i}$.]

Let $v \neq 0$ be an element of $\Ker(\Tr_{k/k^{(d=1)}})$.
It follows from \cite{1}, Lemma 1.3, that $v$ can be uniquely expressed as a $\Q_{p_{k}}$-linear combination of $y_1, \ldots, y_{d_k}$.
\begin{enumerate}
    \item [(1)] The case where $d_{k}$ is even.
\end{enumerate}
When $v$ is expressed as a $\Q_{p_{k}}$-linear combination of $y_1, \ldots, y_{d_k}$, we shall write $c_i$ [respectively, $e_{i}$] for the coefficient of $y_{2i+1}$ [respectively $y_{2i+2}$].
We first show that $y_3, \ldots, y_{d_k} \in \Ker(\Tr_{k/k^{(d=1)}})$.  
To this end, we prove the following claim.
\begin{claim*}
        \begin{enumerate}
       The following assertions hold:
    \item[$(a)$] If $e_i \neq 0$, then $y_{2i+1},\ y_{2i+2} \in \Ker(\Tr_{k/k^{(d=1)}})$.
    
    \item[$(b)$] If $c_i \neq 0$, then $y_{2i+1},\ y_{2i+2} \in \Ker(\Tr_{k/k^{(d=1)}})$.
    
    \item[$(c)$] If $c_i \neq 0$, then $y_{2i+3} \in \Ker(\Tr_{k/k^{(d=1)}})$.
    
    \item[$(d)$] If $c_{i+1} \neq 0$, then $y_{2i+1} \in \Ker(\Tr_{k/k^{(d=1)}})$.
\end{enumerate}
\end{claim*}
We first prove assertion (a).
Suppose that $e_{i} \neq 0$ holds.
It follows from the definition of the map $\Out(G_{k}) \to \Aut(k_{+})$ and Lemma \ref{preservation Ker(Tr)} that the following equality holds:
\begin{align*}
    e_{i}y_{2i+1}=(\varphi_{i})_{+}(v)-v \in \Ker(\Tr_{k/k^{(d=1)}}).
\end{align*}
Since $e_{i} \neq 0$, it follows that $y_{2i+1} \in \Ker(\Tr_{k/k^{(d=1)}})$.
Therefore, it follows from Lemma \ref{preservation Ker(Tr)} that we obtain that
\begin{align*}
    -y_{2i+2}=(\varphi^{\prime}_{i})_{+}(y_{2i+1})-y_{2i+1} \in \Ker(\Tr_{k/k^{(d=1)}}).
\end{align*}
This completes the proof of assertion (a).

Next, we prove assertion (b).
Suppose that $c_{i} \neq 0$.
It follows from the definition of the map $\Out(G_{k}) \to \Aut(k_{+})$ and Lemma \ref{preservation Ker(Tr)} that the following equality holds:
\begin{align*}
    -c_{i}y_{2i+2}=(\varphi^{\prime}_{i})_{+}(v)-v \in \Ker(\Tr_{k/k^{(d=1)}}).
\end{align*}
Since $c_{i} \neq 0$, it follows that $y_{2i+2} \in \Ker(\Tr_{k/k^{(d=1)}})$ holds.
Therefore, it follows from Lemma \ref{preservation Ker(Tr)} that we obtain that
\begin{align*}
    y_{2i+1}=(\varphi_{i})_{+}(y_{2i+2})-y_{2i+2} \in \Ker(\Tr_{k/k^{(d=1)}}).
\end{align*}
This completes the proof of assertion (b).

Next, we prove assertion (c).
Suppose that $c_{i} \neq 0$.
Then it follows from assertion (b) that $y_{2i+1},\ y_{2i+2} \in \Ker(\Tr_{k/k^{(d=1)}})$.
Moreover, it follows from the definition of the map $\Out(G_{k}) \to \Aut_{\Q_{p_{k}}}(k_{+})$ that the following equality holds:
\begin{align*}
    (\varphi^{\prime\prime}_{i})_{+}(v)-v=(c_{i+1}-c_{i})(y_{2i+2}-y_{2i+4}).
\end{align*}
Therefore, it follows from Lemma \ref{preservation Ker(Tr)} that we obtain that
\begin{align*}
    y_{2i+4}=(\varphi^{\prime\prime}_{i})_{+}(y_{2i+1})-y_{2i+1}+y_{2i+2} \in \Ker(\Tr_{k/k^{(d=1)}}).
\end{align*} 
Consequently, it follows from assertion (a) that $y_{2i+3} \in \Ker(\Tr_{k/k^{(d=1)}})$.
This completes the proof of assertion (c).

Next, we prove assertion (d).
Suppose that $c_{i+1} \neq 0$.
Then it follows from assertion (b) that $y_{2i+3},\ y_{2i+4} \in \Ker(\Tr_{k/k^{(d=1)}})$.
Moreover, it follows from the definition of the map $\Out(G_{k}) \to \Aut_{\Q_{p_{k}}}(k_{+})$ that the following equality holds:
\begin{align*}
    (\varphi^{\prime\prime}_{i})_{+}(v)-v=(c_{i+1}-c_{i})(y_{2i+2}-y_{2i+4}).
\end{align*}
Therefore, it follows from Lemma \ref{preservation Ker(Tr)} that we obtain that
\begin{align*}
    y_{2i+2}=-(\varphi^{\prime\prime}_{i})_{+}(y_{2i+3})+y_{2i+3}+y_{2i+4} \in \Ker(\Tr_{k/k^{(d=1)}}).
\end{align*} 
Consequently, it follows from assertion (a) that $y_{2i+1} \in \Ker(\Tr_{k/k^{(d=1)}})$.
This completes the proof of assertion (d).

We now return to the proof of Theorem \ref{determination of Ker(Tr)} in the case where $d_{k}$ is even.
Since $\dim \Ker(\Tr_{k/k^{(d=1)}}) = d_k - 1 \geq 3$ [cf.\ the assumption that $d_{k} \geq 3$ and $d_{k}$ is even],  
it follows that there exists some $1 \leq i \leq g$ such that either $c_i \neq 0$ or $e_i \neq 0$.  
It follows from the above claim by a sort of induction that we have $y_3, \ldots, y_{d_k} \in \Ker(\Tr_{k/k^{(d=1)}})$.

We extend the linearly independent system $\{ y_{3},\ldots,y_{d_{k}}\}$ to a basis $\{ y_{3},\ldots,y_{d_{k}}, w\}$ of $\Ker(\Tr_{k/k^{(d=1)}})$ over $\Q_{p_{k}}$.
We may take $w$ as of the form $\alpha_{1}y_{1} + \alpha_{2}y_{2}$ where $\alpha_{1},\alpha_{2} \in \Q_{p_{k}}$.
If $\alpha_2 = 0$, then there is nothing to prove. 
Suppose that $\alpha_2 \neq 0$.
Let $\varphi$ denote the automorphism of $G_{k}$ defined similarly to ``$\varphi$'' that appears in the proof of the case $d_{k} = 2$.
Then it follows from Lemma \ref{preservation Ker(Tr)} and the definition of the map $\Out(G_{k}) \to \Aut_{\Q_{p_{k}}}(k_{+})$ that the following equality holds:
\begin{align*}
    \alpha_{2}y_{1}=\varphi_{+}(\alpha_{1}y_{1} + \alpha_{2}y_{2})-(\alpha_{1}y_{1} + \alpha_{2}y_{2}) \in \Ker (\Tr_{k/k^{(d=1)}}).
\end{align*}
This completes the proof of the case where $d_{k}$ is even.
\begin{enumerate}
    \item [(2)] The case where $d_{k}$ is odd.
\end{enumerate}
In this case, by applying a similar argument to the argument applied in (1), we obtain $y_{2}, \ldots, y_{d_{k}} \in \Ker(\Tr_{k/k^{(d=1)}})$.
This completes the proof of Theorem \ref{determination of Ker(Tr)}.
\end{proof}

\begin{rem}
    In traditional algorithmic anabelian geometry, non-canonical objects such as generators and relations have not played significant roles. 
    In contrast, Theorem \ref{determination of Ker(Tr)}—by fully exploiting generators and relations in addition to conventional methods—establishes a nontrivial connection between these generators and relations and the ring structure of $k$.
    For these reasons, Theorem \ref{determination of Ker(Tr)} is meaningful in its own right.
\end{rem}

\begin{lem}\label{O_{k^{(d=1)}}}
     Suppose that $p_{k}$ is odd, and that $d_{k}$ is even.
     Let $\varphi$ be the automorphism of $G_{k}$ defined by the following equalities [cf.\ Theorem \ref{generator and relation of group of MLF-type}]:
    \begin{align*}
    \varphi(\sigma)=\sigma,\ \varphi(\tau)=\tau,\ \varphi(x_{2})=x_{2}x_{1},\ \varphi(x_{i})=x_{i}\ (i \neq 2).
    \end{align*}
    Then the following assertions hold:
    \begin{enumerate}
        \item [(1)] For all $n \in \Z \setminus \{0\}$, it holds that $\varphi^{n}_{+}(k^{(d=1)}_{+}) \neq k^{(d=1)}_{+}$.
        \item [(2)] The intersection $\varphi^{\times}(\mathcal{O}_{k^{(d=1)}}^{\times}) \cap \mathcal{O}_{k^{(d=1)}}^{\times}$ is not open in $\mathcal{O}_{k^{(d=1)}}^{\times}$
    \end{enumerate}
\end{lem}

\begin{proof}
   First, we verify assertion (1).
   Since $d_{k}$ is even, it follows from Theorem \ref{determination of Ker(Tr)} that $\Ker(\Tr_{k/k^{(d=1)}}) = \langle y_{1},y_{3},\ldots,y_{d_{k}} \rangle_{\Q_{p_{k}}}$.
    Moreover, since $k^{(d=1)}_{+}$ is torsion-free, and $\Tr_{k/k^{(d=1)}}(x) = d_k x$ for any $x \in k^{(d=1)}_{+}$, it follows that
    \begin{align*}
        \Ker(\Tr_{k/k^{(d=1)}}) \cap k^{(d=1)}_{+} = 0.
    \end{align*}
    Thus, the 1-dimensional sub-$\Q_{p_{k}}$-vector space $k^{(d=1)}_{+}$ has a basis consisting of an element of the form $\alpha_1y_1 + y_2 + \alpha_{3}y_{3}+ \cdots + \alpha_{d_k}y_{d_k}$, where $\alpha_{i}$ is an element of $\Q_{p_{k}}$ for each $i \in \{1,2,\ldots,d_{k}\}$.
    Then one verifies easily from the definition of $\varphi$ that for $n \in \Z \setminus \{0\}$
    \begin{align*} (\varphi^{n})_{+}(\alpha_{1}y_{1}+y_{2}+ \alpha_{3}y_{3}+ \cdots+\alpha_{d_{k}} y_{d_{k}}) = (\alpha_{1}+n)y_{1}+y_{2} +\alpha_{3}y_{3}+\cdots+\alpha_{d_{k}}y_{d_{k}} \notin k^{(d=1)}_{+}.
    \end{align*}
    This completes the proof of assertion (1).

   Next, we verify assertion (2).
   Since $\mathcal{O}_{k^{(d=1)}}^{\times}$ is a profinite group, to show that the subgroup $\varphi^{\times}(\mathcal{O}_{k^{(d=1)}}^{\times}) \cap \mathcal{O}_{k^{(d=1)}}^{\times} \subset \mathcal{O}_{k^{(d=1)}}^{\times}$ is not open in $\mathcal{O}_{k^{(d=1)}}^{\times}$, it suffices to show that it is not a subgroup of finite index.
   Moreover, it follows from \cite{3}, Lemma 1.2, $\rm(\hspace{.18em}i\hspace{.18em})$, $\rm(i\hspace{-.08em}v\hspace{-.06em})$, and the definition of perfection [cf.\ \S 0, Notational conventions, Monoids] that it suffices to show $\varphi_{+}(k^{(d=1)}_{+}) \cap k^{(d=1)}_{+} = 0$ in order to prove that the subgroup 
   \begin{align*}
       \varphi^{\times}(\mathcal{O}_{k^{(d=1)}}^{\times}) \cap \mathcal{O}_{k^{(d=1)}}^{\times} \subset \mathcal{O}_{k^{(d=1)}}^{\times}
   \end{align*}
   is not of finite index. 
    It immediately follows from assertion (1) and the fact that the subspace $k_{+}^{(d=1)}$ is a one-dimensional $\Q_{p_k}$-vector space.  
    Thus, the subgroup $\varphi^{\times}(\mathcal{O}_{k^{(d=1)}}^{\times}) \cap \mathcal{O}_{k^{(d=1)}}^{\times}$ is not of finite index in $\mathcal{O}_{k^{(d=1)}}^{\times}$.
    This completes the proof of assertion (2).
\end{proof}

\begin{dfn}
    We shall say that the MLF $k$ is an \textit{absolutely Galois} MLF if the extension $k/k^{(d=1)}$ is a Galois extension. 
\end{dfn}

\begin{dfn}
    Let $V$ be a $\Q_{p_{k}}$-vector space of finite dimension and $\rho \colon G_{k} \to \Aut_{\Q_{p_{k}}}(V)$ a continuous representation.
    Then we shall say that $\rho$ is \textit{Aut-intrinsically Hodge-Tate} if, for an arbitrary [continuous] automorphism $\alpha$ of $G_{k}$, the composite $\rho \circ \alpha \colon G_{k} \to \Aut_{\Q_{p_{k}}}(V)$ is Hodge-Tate.
\end{dfn}

\begin{dfn}
    Suppose that $k$ is an absolutely Galois MLF.
    Let $\pi \in \mathcal{O}_{k}$  be a uniformizer of $\mathcal{O}_{k}$ and $\sigma$ an element of $\Gal(k/k^{(d=1)})$.
    Then we shall write
    \begin{align*}
        \chi_{\pi,\sigma} \colon G_{k}^{\ab} \stackrel{\rec_{k}^{-1}}{\to} \widehat{k^{\times}} \twoheadrightarrow \mathcal{O}_{k}^{\times} \stackrel{\sigma}{\to} \mathcal{O}_{k}^{\times},
    \end{align*}
   where the second arrow is the projection determined by $\pi$.
\end{dfn}

With the above preparations, we now prove the first main theorem of the present paper.

\begin{thm}\label{Aut-intristic Hodge-Tate}
    Suppose that $p_{k}$ is odd, that $d_{k}$ is even, and that $k$ is an absolutely Galois MLF.
     Then there exist a $\Q_{p_{k}}$-vector space $V$ of dimension $d_{k}$, a continuous representation $\rho \colon G_{k }\to \Aut_{\Q_{p_{k}}}(V)$ that is irreducible, abelian, and crystalline [hence also Hodge-Tate], but not Aut-intrinsically Hodge-Tate.
\end{thm}

\begin{proof}
    Let $\pi \in \mathcal{O}_{k}$ be a uniformizer of $\mathcal{O}_{k}$.
    We write $\rho$ for the continuous representation of $G_{k}$  [necessarily of dimension $d_{k}$] obtained by forming the composite
    \begin{align*}
        G_{k} \twoheadrightarrow G_{k}^{\ab} \stackrel{\chi_{\pi,\id_{k}}}{\to} \mathcal{O}_{k}^{\times} \hookrightarrow \Aut_{\Q_{p_{k}}}(k_{+}),
    \end{align*}
     where the first arrow is the natural surjective continuous homomorphism, and the third arrow is the natural inclusion.
     Then one verifies easily that this continuous representation $\rho$ is irreducible and abelian.
     
     Next, we show that $\rho$ is a crystalline representation.
     Let $F_{\pi}$ be the Lubin-Tate formal group law associated with the Lubin-Tate formal power series $\pi t+t^{{p_{k}}^{f_{k}}}$.
     Then $\rho$ is isomorphic to the continuous $p_{k}$-adic representation induced by the action on the $p_{k}$-adic Tate module of $F_{\pi}$ [cf., e.g., \cite{12}, $\rm I\hspace{-.15em}I\hspace{-.15em}I$, \S A.4, Proposition 4].
     Since $F_{\pi}$ is defined over $\mathcal{O}_{k}$, the corresponding $p_{k}$-divisible group is also defined over $\mathcal{O}_{k}$.
     It is well-known [cf.\ \cite{15}, Theorem 6.2, $\rm\hspace{.18em}i\hspace{.18em})$] that the continuous $p_{k}$-adic representation induced by the action on the $p_{k}$-adic Tate module of $F_{\pi}$ is crystalline.
     Thus, the representation $\rho$ is a crystalline representation.
     
     Next, we show by contradiction that the continuous representation $\rho \circ \varphi$ is not Hodge–Tate, where $\varphi$ is the automorphism of $G_{k}$ defined in the statement of Lemma \ref{O_{k^{(d=1)}}}.
     Let us recall that it follows immediately from the various definitions involved that the composite
     \begin{align*}
         \mathcal{O}_{k}^{\times} \stackrel{\rec_{k}}{\hookrightarrow} G_{k}^{\ab} \stackrel{\chi_{\pi,\id_{k}}}{\to} \mathcal{O}_{k}^{\times}
     \end{align*}
     is an automorphism that restricts to an automorphism of the subgroup $\mathcal{O}_{k^{(d=1)}}^{\times} \subset \mathcal{O}_{k}^{\times}$.
    We write $\varphi^{\ab} \colon G_{k}^{\ab} \stackrel{\sim}{\longrightarrow} G_{k}^{\ab}$ for the automorphism of the topological group $G_{k}^{\ab}$ induced by $\varphi$.
     It follows from the assumption of the contradiction that $\rho \circ \varphi$ is Hodge–Tate.
     Thus, it follows from \cite{11}, Lemma 1.9, that there exists an open subgroup $U \subset \mathcal{O}_{k^{(d=1)}}^{\times}$ whose image by the composite
     \begin{align*}
         \mathcal{O}_{k^{(d=1)}}^{\times} \hookrightarrow \mathcal{O}_{k}^{\times} \hookrightarrow G_{k}^{\ab} \to \mathcal{O}_{k}^{\times},
     \end{align*}
     where the first arrow is a natural injective homomorphism, the second arrow is the injective homomorphism obtained by restricting the local reciprocity homomorphism of $k$, and the third arrow is $\chi_{\pi,\id_{k}} \circ \varphi^{\ab}$, is contained in the subgroup $\mathcal{O}_{k^{(d=1)}}^{\times} \subset \mathcal{O}_{k}^{\times}$.
     Thus, it follows from the various definitions involved that $\varphi^{\times}(U) \subset \mathcal{O}_{k^{(d=1)}}^{\times}$ and is an open subgroup of $\mathcal{O}_{k^{(d=1)}}^{\times}$.
     In particular, we obtain that $\varphi^{\times}(U) \subset \varphi^{\times}(\mathcal{O}_{k^{(d=1)}}^{\times}) \cap \mathcal{O}_{k^{(d=1)}}^{\times}$.
    However, this contradicts Lemma \ref{O_{k^{(d=1)}}}, (2), since $\varphi^{\times}(U) \subset \mathcal{O}_{k^{(d=1)}}^{\times}$ is an open subgroup of $\mathcal{O}_{k^{(d=1)}}^{\times}$.
     Thus, the composite $\rho \circ \varphi \colon G_{k} \to \Aut_{\Q_{p_{k}}}(k_{+})$ is not Hodge-Tate.
     This completes the proof of Theorem \ref{Aut-intristic Hodge-Tate}.
\end{proof}

\begin{rem}\label{odd case}
    In the case where $d_k$ is odd, a proof analogous to that of Theorem \ref{Aut-intristic Hodge-Tate} would require an element $\alpha \in \Aut(G_{k})$ such that $\alpha_{+}(y_1) \neq y_1$.
    However, at the time of writing, it is not clear to the author whether or not such an element exists.
    For this reason, the assumption that $d_k$ is even is imposed in Theorem \ref{Aut-intristic Hodge-Tate}.

\end{rem}

\begin{rem}\label{Aut-intristic  Hodge-Tate vs Hodge-Tate}
    Suppose that $d_{k}=1$ and that $p_{k}$ is odd. 
    As observed in the proof of Theorem \ref{Aut-intristic Hodge-Tate}, there exists an irreducible, abelian, and crystalline [hence, Hodge-Tate] one-dimensional continuous representation $\rho_{0} \colon G_k \to \Aut_{\Q_{p_k}}(V)$.
    As shown in \cite{11}, Theorem A, when a continuous  $p_{k}$-adic representation $\rho \colon G_k \to \Aut_{\Q_{p_k}}(V)$ is one-dimensional, or, even, two-dimensional and reducible, the property that $\rho$ is Hodge-Tate is equivalent to the property that $\rho$ is Aut-intrinsically Hodge-Tate.
    Thus, this representation $\rho_{0} \colon G_k \to \Aut_{\Q_{p_k}}(V)$ is Aut-intrinsically Hodge-Tate.
    
    On the other hand, it is well-known [cf., e.g., the discussion given at the final portion of \cite{6}, Chapter V\hspace{-1.2pt}I\hspace{-1.2pt}I, \S5] that there exists an outer automorphism of $G_k$ that is not induced by any field automorphism of $k$.
    Thus, it follows from \cite{4}, Corollary 3.4, that there exists a continuous representation of $G_{k}$ that is Hodge-Tate but not Aut-intrinsically Hodge-Tate.
It follows that the following inclusion holds:
\begin{align*}
    \emptyset \neq \{ \text{Aut-intrinsically Hodge-Tate}\}\subsetneq \{\text{Hodge-Tate}\}.
\end{align*}
\end{rem}

\begin{rem}
   Let $\chi \colon G_{k} \to \Q_{p_{k}}^{\times}$ be a nontrivial continuous character which factors through $\Gal(K/k)$ for some totally ramified finite extension $K/k$ of degree greater than one.
Then we observe that the representation $\rho^{\prime}$ defined by $\chi$ is de Rham [hence also Hodge-Tate] but not crystalline as follows.
In fact, the de Rham-ness of $\rho^{\prime}$ follows from the facts that a continuous $p_k$-adic representation is de Rham if and only if its restriction to an open subgroup is de Rham, together with the fact that the trivial representation is de Rham.
Next, we show that the representation $\rho^{\prime}$ is not crystalline.
Let us write $G_{K}\stackrel{\mathrm{def}}{=} \Gal(\overline{k}/K)$ for the absolute Galois group of $K$ determined by $\overline{k}$.
Since $K/k$ is a totally ramified extension, we have $B_{\mathrm{cris}}^{G_K} \cong k_0$, where $B_{\mathrm{cris}}$ denotes the crystalline period ring, and $k_0$ is the maximal unramified subextension of $k/k^{(d=1)}$.
Moreover, by the definition of $\chi$, we have $\Q_{p_k}(\chi)^{G_K} = \Q_{p_k}(\chi)$ and $\Q_{p_k}(\chi)^{\Gal(K/k)} = 0$, where $\Q_{p_k}(\chi)$ denotes the $\Q_{p_k}[G_k]$-module whose underlying $\Q_{p_k}$-vector space is $\Q_{p_k}$ and on which $G_k$ acts via $\chi$.
Thus, we have
\begin{align*}
    (\Q_{p_{k}}(\chi) \otimes_{\Q_{p_{k}}} B_{\mathrm{cris}})^{G_{k}}
    \cong &((\Q_{p_{k}}(\chi) \otimes_{\Q_{p_{k}}} B_{\mathrm{cris}})^{G_{K}})^{\Gal(K/k)} \cong (\Q_{p_{k}}(\chi) \otimes_{\Q_{p_{k}}} k_{0})^{\Gal(K/k)}=0.
\end{align*}
In particular, the representation $\rho^{\prime}$ is not crystalline.

We note that any one-dimensional continuous linear representation that is Hodge-Tate is necessarily Aut-intrinsically Hodge-Tate [cf.\ \cite{11}, Theorem A]. In particular, the representation $\rho^{\prime}$ is Aut-intrinsically Hodge-Tate but not crystalline.
Thus, there is no inclusion relation in either direction between the classes of Aut-intrinsically Hodge-Tate representations and crystalline representations.
\end{rem}

\begin{rem}
    An interesting problem related to Aut-intrinsic Hodge-Tate-ness is whether, given an automorphism $\alpha \colon G_{k} \stackrel{\sim}{\longrightarrow} G_{k}$, one can suitably formulate, by means of $\alpha$, the difference between the class of continuous representations of $G_k$ that are Hodge-Tate and the class of continuous representations of $G_k$ that are Hodge-Tate but cease to be so after composition with $\alpha$.
    However, at the time of writing the present paper, this problem remains unexplored.
\end{rem}

\clearpage

\section{On the non-normality of the field-theoretic subgroups of the outer automorphism groups of the absolute Galois groups of absolutely Galois MLFs}

In the present \S2, let $k$ be an MLF, $\overline{k}$ an algebraic closure of $k$, and $G$ a group of MLF-type.
In the present \S2, we discuss the outer automorphism groups of the absolute Galois groups of certain MLFs.
In particular, we show that if $p_{k}$ is odd, $d_{k}$ is greater than one, and $k$ is an absolutely Galois MLF, i.e., the extension $k/k^{(d=1)}$ is a Galois extension, then the field-theoretic subgroup $\Aut(k) \subset \Out(G_{k})$ is not normal.
This result generalizes certain theorems in \cite{1} and \cite{2}.

\begin{dfn}\label{field-theoritic subgroup}
    We shall refer to the image of the natural injective group homomorphism $\Aut(k) \hookrightarrow \Out(G_{k})$ [cf., e.g., \cite{3}, Proposition 2.1] as the \textit{field-theoretic subgroup} of $\Out(G_{k})$.
\end{dfn}

\begin{rem}\label{group-theoretic but not ring theoretic isom}
    It is well-known that the isomorphism class of the field $k$ itself cannot be reconstructed group-theoretically from $G_k$ in general.
    For example, let $p$ be an odd prime number, $k_0 \stackrel{\mathrm{def}}{=} \Q_p(\zeta_p, \sqrt[p]{p})$ and $k_1 \stackrel{\mathrm{def}}{=} \Q_p(\zeta_p, \sqrt[p]{p+1})$, where $\zeta_{p}$ is a primitive $p$-th root of unity.
    Fix an algebraic closure $\overline{k_0}$ of $k_0$ and an algebraic closure $\overline{k_1}$ of $k_1$.
    Then it follows from the main theorem of \cite{10} that the group $G_{k_0}\stackrel{\mathrm{def}}{=}\Gal(\overline{k_{0}}/k_{0})$ is isomorphic to the group $G_{k_1}\stackrel{\mathrm{def}}{=}\Gal(\overline{k_{1}}/k_{1})$, while the field $k_0$ is not isomorphic to the field $k_1$.
    
    Furthermore, the following observation illustrates, in a stronger sense, that a group isomorphism $G_{k_0} \cong G_{k_1}$ is not compatible with their ring structures.
    Let $\varphi \colon G_{k_1} \stackrel{\sim}{\longrightarrow} G_{k_0}$ be an arbitrary group isomorphism, and let $\Phi \colon \Out(G_{k_1}) \stackrel{\sim}{\longrightarrow} \Out(G_{k_0})$ denote the induced group isomorphism.
    Then it follows from \cite{2}, Corollary 8.6, that $\Phi(\Gal(k_1/\Q_{p}(\zeta_{p}))) \not\subset N_{\Out(G_{k_0})}(\Aut(k_0))$.
    In particular, the subgroup $\Aut(k_{0}) \subset \Out(G_{k_{0}})$ is not a normal subgroup of $\Out(G_{k_{0}})$.
\end{rem}

It follows from Lemma \ref{O_{k^{(d=1)}}}, (2), that there exists an outer automorphism $\alpha \in \Out(G_{k})$ such that the automorphism $\alpha^{\times} \colon k^{\times} \stackrel{\sim}{\longrightarrow} k^{\times}$ induced by the mono-anabelian reconstruction algorithm does not preserve the subgroup $(k^{(d=1)})^{\times} \subset k^{\times}$.
In particular, the subgroup $(k^{(d=1)})^{\times} \subset k^{\times}$  cannot be reconstructed group-theoretically and functorially from the group $G_{k}$ in general.
However, we will show that if $k$ is absolutely Galois and $\alpha \in N_{\Out(G_{k})}(\Aut(k))$, then the induced isomorphism $\alpha^{\times}$ preserves the subgroup $(k^{(d=1)})^{\times} \subset k^{\times}$ [cf.\ Theorem \ref{Q^{×}-preserving} below].

\begin{dfn}
    Let $\Pi$ be a profinite group.
    Then we shall say that $\Pi$ is \textit{slim} if for every open subgroup $\Gamma$ of $\Pi$, the equality $Z_{\Pi}(\Gamma) = {1}$ holds.
\end{dfn}
A group of MLF-type is an example of a slim profinite group [cf., e.g., \cite{3}, Proposition 1.8].

\begin{lem}\label{slim profinite group}
    Let $\Pi$ be a slim profinite group, and let $\Gamma \subset \Pi$ be an open normal subgroup of $\Pi$.
    We write $\Aut_{\Gamma}(\Pi)$ for the subgroup of $\Aut(\Pi)$ consisting of elements $\varphi$ such that $\varphi(\Gamma) = \Gamma$.
    Then the following assertions hold:
    \begin{enumerate}
    \item [(1)] The restriction map
    \begin{align*}
        \Aut_{\Gamma}(\Pi) \to \Aut(\Gamma);\ \varphi \mapsto \varphi|_{\Gamma}
    \end{align*}
    is an injective group homomorphism.
    \item [(2)]  The outer action $\Pi/\Gamma \to \Out(\Gamma)$, which is defined by the natural exact sequence of groups
    \begin{align*}
        1 \to \Gamma \to \Pi \to \Pi/\Gamma \to 1,
    \end{align*}
    is injective.
    \item [(3)] We regard $\Pi/\Gamma$ as a subgroup of $\Out(\Gamma)$ via the outer action $\Pi/\Gamma \hookrightarrow \Out(\Gamma)$ [cf.\ assertion (2)] and $\Aut_{\Gamma}(\Pi)/\Inn(\Gamma)$ as a subgroup of $\Out(\Gamma)$ via the natural injection $\Aut_{\Gamma}(\Pi)/\Inn(\Gamma) \hookrightarrow \Out(\Gamma)$ [cf.\ assertion (1)].
    Then the following equality holds:
    \begin{align*}
        N_{\Out(\Gamma)}(\Pi/\Gamma)=\Aut_{\Gamma}(\Pi)/\Inn(\Gamma).
    \end{align*}
    \end{enumerate} 
\end{lem}

\begin{proof}
    First, we verify assertion (1).
    Let $\varphi \in \Aut_{\Gamma}(\Pi)$.
    Since $\Pi$ is slim, the group homomorphism $\Pi \to \Aut(\Gamma)$ by conjugation is injective.
    Moreover, it follows from various definitions involved that the following diagram commutes:
    \[
\begin{tikzcd}
\Pi \arrow[r, "\subset"] \arrow[d, "\varphi"]
&\Aut(\Gamma) \arrow[d, "\cong"] \\
\Pi \arrow[r, "\subset"]
& || \Aut(\Gamma),
\end{tikzcd}
    \]
    where the right-hand vertical arrow is the group homomorphism given conjugation by $\varphi|_{\Gamma}$.
   Thus, if $\varphi|_{\Gamma} = \id_{\Gamma}$, then it follows from the commutativity of this diagram and the injectivity of $\Pi \to \Aut(\Gamma)$ that $\varphi = \id_{\Pi}$.
   This completes the proof of assertion (1).
    
    Next, we verify assertion (2).
    Let $\pi \in \Pi$ be an element.
    We write $i_{\pi}$ for the inner automorphism of $\Pi$ determined by $\pi$.
    Suppose that the automorphism $i_{\pi}|_{\Gamma}$ of $\Gamma$ is an inner automorphism.
    Then there exists an element $\gamma_{0} \in \Gamma$ such that $i_{\pi}|_{\Gamma} = i_{\gamma_{0}}|_{\Gamma}$, where $i_{\gamma_{0}}$ denotes the inner automorphism of $\Pi$ determined by $\gamma_{0}$.
    Thus, it follows from the slimness of $\Pi$ and the injectivity of the map $\Aut_{\Gamma}(\Pi) \to \Aut(\Gamma)$, $\varphi \mapsto \varphi|_{\Gamma}$ [cf.\ assertion (1)], that $\pi = \gamma_{0} \in \Gamma$.
    This completes the proof of assertion (2).
    
    Next, we verify assertion (3).
    Since $\Pi$ is slim, and slimness is inherited by open subgroups, it follows that $\Gamma$ is also slim.
    In particular, the injection $\Pi/\Gamma \hookrightarrow \Out(\Gamma)$ of assertion (2) splits as the isomorphism $\Pi/\Gamma \stackrel{\sim}{\longrightarrow} \Inn(\Pi)/\Inn(\Gamma)$ and the inclusion $\Inn(\Pi)/\Inn(\Gamma) \hookrightarrow \Out(\Gamma)$.
    Thus, it is immediate that, to verify assertion (3), it suffices to show that $N_{\Aut(\Gamma)}(\Inn(\Pi)) = \Aut_{\Gamma}(\Pi)$.
    The inclusion $\Aut_{\Gamma}(\Pi) \subset N_{\Aut(\Gamma)}(\Inn(\Pi))$ is trivial.

    We now show the inclusion $N_{\Aut(\Gamma)}(\Inn(\Pi)) \subset \Aut_{\Gamma}(\Pi)$.
    Take an element $\varphi \in N_{\Aut(\Gamma)}(\Inn(\Pi))$.
    For each element $\pi \in \Pi$, there exists $\pi^{\prime} \in \Pi$ such that the following equality holds:
   \begin{align*}
       \varphi \circ (i_{\pi}|_{\Gamma}) \circ \varphi^{-1}=i_{\pi^{\prime}}|_{\Gamma},
   \end{align*}
   where we write $i_{\pi}$, $i_{\pi^{\prime}}$ for the inner automorphisms of $\Pi$ determined by $\pi$, $\pi^{\prime}$ respectively.
   Since $\Gamma$ is slim, the element $\pi^{\prime}$ is uniquely determined by $\pi$.
   Moreover, it is easy to verify that if $\pi \in \Gamma$, then $\pi^{\prime} = \varphi(\pi)$.
   Furthermore, it is also easy to verify that if we define the map $\psi \colon \Pi \to \Pi$ by $\psi(\pi) \stackrel{\mathrm{def}}{=} \pi^{\prime}$, then $\psi$ is an automorphism of $\Pi$ that satisfies $\psi|_{\Gamma} = \varphi$.
Thus, the desired inclusion follows.
\end{proof}

 By applying Lemma \ref{slim profinite group}, we obtain the following result:

\begin{prop}\label{ab. gal normalizer}
    Suppose that $k$ is an absolutely Galois MLF.
    We shall write $\Aut_{G_{k}} (G_{k^{(d=1)}})$ for the subgroup of $\Aut(G_{k^{(d=1)}})$ consisting of elements $\varphi$ such that $\varphi(G_{k}) = G_{k}$.
    Then we have, by considering restrictions, a natural homomorphism
    \begin{align*}
        \Aut_{G_{k}}(G_{k^{(d=1)}}) \to \Aut(G_{k}).
    \end{align*}
    Then the following assertions hold:
    \begin{enumerate}
        \item [(1)] The homomorphism $\Aut_{G_{k}}(G_{k^{(d=1)}}) \to \Aut(G_{k})$ is injective.
        In particular, we also have an injection $\Aut_{G_{k}}(G_{k^{(d=1)}})/\Inn(G_{k}) \hookrightarrow \Out(G_{k})$.
        Let us regard $\Aut_{G_{k}}(G_{k^{(d=1)}})$, $\Aut_{G_{k}}(G_{k^{(d=1)}})/\Inn(G_{k})$ as subgroups of $\Aut(G_{k})$, $\Out(G_{k})$ by means of these injections, respectively:
    \begin{align*}
        \Aut_{G_{k}}(G_{k^{(d=1)}}) \subset \Aut(G_{k}),\ \ \Aut_{G_{k}}(G_{k^{(d=1)}})/\Inn(G_{k}) \subset \Out(G_{k}).
    \end{align*}
    \item [(2)] The natural homomorphisms $G_{k^{(d=1)}} \to \Inn(G_{k^{(d=1)}}) \hookrightarrow \Aut(G_{k^{(d=1)}})$ determine an isomorphism
    \begin{align*}
        G_{k^{(d=1)}}/G_{k} \stackrel{\sim}{\longrightarrow} \Inn(G_{k^{(d=1)}})/\Inn(G_{k}).
    \end{align*}
     Let us identify $G_{k^{(d=1)}}/G_{k}$ with $\Inn(G_{k^{(d=1)}})/\Inn(G_{k})$  by means of this isomorphism
     \begin{align*}
G_{k^{(d=1)}}/G_{k}=\Inn(G_{k^{(d=1)}})/\Inn(G_{k})\ \  (\subset \Aut_{G_{k}}(G_{k^{(d=1)}})/\Inn(G_{k}) \subset \Out(G_{k})).
     \end{align*}
     \item [(3)] The following equality of subgroups of $\Out(G_{k})$ holds:
      \begin{align*}
          N_{\Out(G_{k})}(\Aut(k))=\Aut_{G_{k}}(G_{k^{(d=1)}})/\Inn(G_{k}).
      \end{align*}
    \end{enumerate}
\end{prop}

\begin{proof}
   Assertion (1) follows directly from Lemma \ref{slim profinite group}, (1), and \cite{3}, Lemma 1.8.
   Assertion (2) follows directly from \cite{3}, Lemma 1.8.
   Assertion (3) follows directly from Lemma \ref{slim profinite group}, (3), \cite{3}, Lemma 1.8, and the commutativity of the following diagram:
   \[
   \begin{tikzcd}
       G_{k^{(d=1)}}/G_{k} \arrow[r, "\subset"] \arrow[d, "\cong"]
&\Out(G_{k})\\
\Aut(k) \arrow[ru, swap, "\subset"],
   \end{tikzcd}
\]
   where the upper horizontal arrow is the injective homomorphism defined in the statement of assertion (2), the diagonal arrow is the natural injective homomorphism, and the left-hand vertical arrow is the natural isomorphism [cf.\ our assumption that $k$ is absolutely Galois].
\end{proof}

\begin{dfn}
     Let $\alpha$ be an element of $\Out(G_{k})$.
     We write $\alpha^{\times} \colon k^{\times} \stackrel{\sim}{\longrightarrow} k^{\times}$ for the automorphism induced from $\alpha$ by the mono-anabelian reconstruction algorithm.
     We shall say that $\alpha$ is $\Q_{p_{k}}^{\times}$-\textit{characteristic} if the equality $\alpha^{\times}((k^{(d=1)})^{\times})= (k^{(d=1)})^{\times}$ holds. 
        We shall say that $\alpha$ is $\Q_{p_{k}}^{\times}$-\textit{preserving} if $\alpha$ is $\Q_{p_{k}}^{\times}$-characteristic, and the equality $\alpha^{\times}|_{(k^{(d=1)})^{\times}}=\id_{(k^{(d=1)})^{\times}}$ holds.
\end{dfn}

\begin{lem}\label{d(G)=1 type}
 If $d_{k}=1$, then every element of $\Out(G_{k})$ is $\Q_{p_{k}}^{\times}$-preserving.
\end{lem}

\begin{proof}
  Let $k(G_{k})$ be the topological field defined in \cite{2}, Definition 5.2 [Note that if $d_{k} = 1$, then $d(G_{k}) = 1$ --- cf. \cite{3}, Proposition 3.6]. 
By definition, the group $k(G_{k})^{\times}$ is the multiplicative group of the field $k(G_{k})$.
Moreover, it follows immediately from \cite{2}, Theorem 5.4, that $k(G_{k})$ is canonically isomorphic to $k$ as a topological field.
It follows from the functoriality of the mono-anabelian reconstruction algorithm that the following diagram commutes:
\[
   \begin{tikzcd}
       \Out(G_{k}) \arrow[r,] \arrow[d,]
&\Aut(k(G_{k})^{\times})\\
\Aut(k(G_{k})) \arrow[ru, swap, "\subset"].
   \end{tikzcd}
\]
Since $\Aut(k(G_{k})) \cong \Aut(k) = 1$ and the above diagram commutes, the homomorphism $\Out(G_{k}) \to \Aut(k(G_{k})^{\times})$ is trivial.
This completes the proof of Lemma \ref{d(G)=1 type}.
\end{proof}

\begin{thm}\label{Q^{×}-preserving} Let $\alpha$ be an element of $\Out(G_{k})$.
    Suppose that $k$ is an absolutely Galois MLF.
        If $\alpha \in N_{\Out(G_{k})}(\Aut(k))$, then $\alpha$ is $\Q_{p_{k}}^{\times}$-preserving.
\end{thm}

\begin{proof}
   This assertion follows immediately from Proposition \ref{ab. gal normalizer}, (3), Lemma \ref{d(G)=1 type}, the commutative diagram appearing in the statement of \cite{3}, Lemma 1.7, $\rm(i\hspace{-.08em}i\hspace{-.08em}i)$, and the functoriality of the transfer map.
\end{proof}

The second main theorem of the present paper --- namely, the non-normality of the field-theoretic subgroup of $\Out(G_k)$ --- follows from some results in \S1 and the techniques developed in \cite{1} under the assumptions that $d_k$ is even, $p_k$ is odd, and $k$ is absolutely Galois.

\begin{thm}\label{non-normality when d_{k} is even}
    Suppose that $k$ is an absolutely Galois MLF, that $p_{k}$ is odd, and that $d_{k}$ is even.
    Then the set of $\Out(G_{k})$-conjugates of the subgroup $\Aut(k) \subset \Out(G_{k})$ is infinite.
    In particular, the subgroup $\Aut(k) \subset \Out(G_{k})$ is not a normal subgroup of $\Out(G_{k})$.
\end{thm}

\begin{proof}
    Let $\varphi$ be the automorphism of $G_{k}$ defined in the statement of Lemma \ref{O_{k^{(d=1)}}}.
    Then it follows from Theorem \ref{Q^{×}-preserving}, and Lemma \ref{O_{k^{(d=1)}}}, (2), that $\varphi^n \notin N_{\Out(G_k)}(\Aut(k))$ for all integer $n \neq 0$.
    In particular, the set of $\Out(G_{k})$-conjugates of the subgroup $\Aut(k) \subset \Out(G_{k})$ is infinite.
\end{proof}

\begin{rem}\label{various non-normality results}
    Theorem \ref{non-normality when d_{k} is even} is a generalization of the main theorem of \cite{1} and one of the main results of \cite{2}, which appeared in Remark~\ref{group-theoretic but not ring theoretic isom}.  
The method of the proof of Theorem \ref{non-normality when d_{k} is even} is essentially the same as the method of the proof that was applied in \cite{1} and is substantially different from the approach taken in \cite{2}.
\end{rem}

\begin{rem}\label{ring theoretic deta}
Suppose that the subgroup $G_k \subset G_{k^{(d=1)}}$ is a characteristic open subgroup.
For example, if $k/k^{(d=1)}$ is an abelian extension, then the subgroup $G_k \subset G_{k^{(d=1)}}$ is characteristic and open [cf.\ \cite{2}, Theorem 6.3, $\rm(\hspace{.18em}i\hspace{.18em})$].
Then, by Proposition \ref{ab. gal normalizer}, (2), (3), we obtain the following exact sequence of groups:
\begin{align*}
    1 \to \Aut(k) \to N_{\Out(G_k)}(\Aut(k)) \to \Out(G_{k^{(d=1)}}) \to 1.
\end{align*}
From the existence of this exact sequence and the fact that $\Out(G_{k^{(d=1)}}) \neq 1$ for odd residue characteristic $p_k$ [cf., e.g., the discussion given at the final portion of \cite{6}, Chapter V\hspace{-1.2pt}I\hspace{-1.2pt}I, \S5], we deduce that $\Aut(k) \subsetneq N_{\Out(G_k)}(\Aut(k))$.
Put another way, there exists an element $\alpha \in N_{\Out(G_k)}(\Aut(k))$ such that $\alpha \notin \Aut(k)$.
As discussed in the discussion following Remark \ref{group-theoretic but not ring theoretic isom}, in general, the subgroup $(k^{(d=1)})^{\times} \subset k^{\times}$ cannot be reconstructed from the group $G_k$ in a purely group-theoretic and functorial manner.  
Put another way, the subgroup $(k^{(d=1)})^{\times} \subset k^{\times}$ depends on the ring structure of $k$ in an essential way.
On the other hand, by Theorem \ref{Q^{×}-preserving}, if $k$ is an absolutely Galois MLF and $\alpha \in N_{\Out(G_k)}(\Aut(k))$, then the induced isomorphism $\alpha^{\times}$ preserves the subgroup $(k^{(d=1)})^{\times} \subset k^{\times}$ and $\alpha^{\times}|_{(k^{(d=1)})^{\times}}=\id_{(k^{(d=1)})^{\times}}$.
It follows from this discussion that every element $\alpha \in N_{\Out(G_k)}(\Aut(k))$, which is not necessarily contained in $\Aut(k)$, preserves certain structures of ring-theoretic nature.
\end{rem}

\begin{rem}
    Let $\alpha$ be an element of $N_{\Out(G_k)}(\Aut(k))$. 
Suppose that $k$ is an absolutely Galois MLF.
As is shown in \cite{1}, Lemma 3.3, it holds that $\alpha_{+}|_{k^{(d=1)}_{+}} = \id_{k^{(d=1)}_{+}}$.  

In what follows, we provide an alternative proof of this fact, which is stated in \cite{1}, by means of Theorem \ref{Q^{×}-preserving}.
What we need to show is that, for $\alpha \in \Out(G_{k})$, if $\alpha$ is $\Q_{p_{k}}^{\times}$-preserving, then it is also $(\Q_{p_{k}})_{+}$-preserving.
Since $\alpha$ is $\Q_{p_{k}}^{\times}$-preserving, we have $\alpha^{\times}|_{(k^{(d=1)})^{\times}}=\id_{(k^{(d=1)})^{\times}}$.
   In particular, we have $\alpha^{\times}|_{\mathcal{O}^{\times}_{k^{(d=1)}}}=\id_{\mathcal{O}^{\times}_{k^{(d=1)}}}$.
   Therefore, on $2p_{k}\mathcal{I}_{k}$, we have $\alpha_{+} = \id_{2p_{k}\mathcal{I}_{k}}$.
   Since, for an arbitrary $x \in k_{+}$, 
     there exists an integer $n$ such that $p_{k}^{n}x \in 2p_{k}\mathcal{I}_{k}$ [cf.\ \cite{3}, Lemma 1.2, $\rm(\hspace{-.06em}v\hspace{-.08em}i)$], we conclude that $\alpha_{+}=\id_{k_{+}}$.
\end{rem}
We now prepare to prove the non-normality of the field-theoretic subgroup of the outer automorphism group of the absolute Galois group of $k$ in the case where $d_k$ is odd, $k$ is absolutely Galois, and $p_{k}$ is odd.

In the remainder of the present paper, we assume that the prime numbers $p_{k}$ and $p(G)$ are odd.

Since $G$ is a topologically finitely generated profinite group [cf.\ Theorem \ref{generator and relation of group of MLF-type}], for any integer $n > 0$, there exist only finitely many subgroups of $G$ of index $n$.
In particular, the set of all characteristic subgroups of $G$ of finite index forms a fundamental system of open neighborhoods of $1 \in G$.
Hence, it follows from \cite{19}, Proposition 4.4.3, that $\Aut(G)$ and $\Out(G)$ have natural profinite group structures.
In the remainder of the present paper, we endow $\Aut(G)$ and $\Out(G)$ with these profinite group structures.

Let us consider the group homomorphism $\Out(G) \to \Aut_{\Q_{p(G)}}(k_{+}(G))$ defined via the mono-anabelian reconstruction algorithm [cf.\ \cite{3}, Proposition 3.11, $\rm(i\hspace{-.08em}v\hspace{-.06em})$, \cite{1}, Lemma 1.2].
We equip $ \Aut_{\Q_{p(G)}}(k_{+}(G))$ with the topology induced by the $p(G)$-adic topology of $k_{+}(G)$, i.e., the relative topology with respect to
\begin{align*}
    \Aut_{\Q_{p(G)}}(k_{+}(G)) \subset \End_{\Q_{p(G)}}(k_{+}(G)) \cong \Q_{p(G)}^{d(G)^{2}},
\end{align*}
or, equivalently, the compact open topology defined by the $p(G)$-adic topology of $k_{+}(G)$.
\begin{lem}\label{closed map}
    The action $\mathrm{Aut}(G) \curvearrowright k_{+}(G)$ which is defined via the mono-anabelian reconstruction algorithm is continuous.  
    In particular, the induced maps $\Aut(G) \to \Aut_{\Q_{p(G)}}(k_{+}(G))$ and $\Out(G) \to \Aut_{\Q_{p(G)}}(k_{+}(G))$ are continuous and closed.
\end{lem}

\begin{proof}
    Since $k_{+}(G)$ is a locally compact Hausdorff space, the topological spaces $\Aut(G)$ and $\Out(G)$ are compact, and $\Aut_{\Q_{p(G)}}(k_{+}(G))$ is Hausdorff, to verify Lemma \ref{closed map}, it suffices to show that the action $\mathrm{Aut}(G) \curvearrowright k_{+}(G)$ defined via the mono-anabelian reconstruction algorithm is continuous.
    Moreover, since the map $\Aut(\mathcal{O}^{\times}(G)) \to \Aut(k_{+}(G))$ induced by the functoriality of perfection [cf.\ \S 0, Notational conventions, Monoids] is continuous with respect to the compact-open topology on $\Aut(\mathcal{O}^{\times}(G))$, to verify Lemma \ref{closed map}, it suffices to show that the action $\Aut(G) \curvearrowright \mathcal{O}^{\times}(G)$ defined via the mono-anabelian reconstruction algorithm is continuous.
    
    By construction, the map $\Aut(G) \to \Aut(\mathcal{O}^{\times}(G))$ factors through the composite $\Aut(G) \to \Aut_{\mathcal{O}^{\times}(G)}(G^{\mathrm{ab}}) \to \Aut(\mathcal{O}^{\times}(G))$, where we write $\Aut_{\mathcal{O}^{\times}(G)}(G^{\mathrm{ab}})$ for the subgroup of $\Aut(G^{\mathrm{ab}})$ consisting of elements that preserve the subgroup $\mathcal{O}^{\times}(G) \subset G^{\mathrm{ab}}$.
    We equip $\Aut(G^{\mathrm{ab}})$ with the profinite topology as in the case of $\Aut(G)$. 
    The subgroup $\Aut_{\mathcal{O}^{\times}(G)}(G^{\mathrm{ab}})$ is endowed with the subspace topology induced from $\Aut(G^{\mathrm{ab}})$.
    It is clear that the map $\Aut(G) \to \Aut_{\mathcal{O}^{\times}(G)}(G^{\mathrm{ab}})$ is continuous.
   Moreover, since the profinite topology and the compact-open topology on $\mathrm{Aut}(G^{\mathrm{ab}})$ coincide [cf., e.g., \cite{19}, Theorem 4.4.2], the map $\Aut_{\mathcal{O}^{\times}(G)}(G^{\mathrm{ab}}) \to \Aut(\mathcal{O}^{\times}(G))$, which is defined by $\varphi \mapsto \varphi|_{\mathcal{O}^{\times}(G)}$, is continuous.
    Thus, the action $\Aut(G) \curvearrowright \mathcal{O}^{\times}(G)$ is continuous.
    This completes the proof of Lemma \ref{closed map}.
\end{proof}

When $d_{k}>1$, let us consider the basis $y_{1},\ldots,y_{d_{k}}$ of $k_{+}$ over $\Q_{p_{k}}$ corresponding to the basis of $k_{+}(G_{k})$ that appeared in the discussion preceding Lemma \ref{preservation Ker(Tr)} via the isomorphism of topological groups $k_{+}(G_{k}) \stackrel{\sim}{\longrightarrow} k_{+}$ of \cite{3}, Proposition 3.11, $\rm(i\hspace{-.08em}v\hspace{-.06em})$.
In the remainder of the present paper, we always equip $k_{+}$ with this basis, which allows us to identify $\Aut_{\Q_{p_{k}}}(k_{+})$ with $\GL_{d_{k}}(\Q_{p_{k}})$.
We equip $\GL_{d_{k}}(\Q_{p_{k}})$ with its natural $p_{k}$-adic Lie group structure.

\begin{lem}\label{the centralizer of Aut(k) in GL}
    Suppose that $k$ is an absolutely Galois MLF and that $d_{k}>1$.
    Then the following assertions hold:
    \begin{enumerate}
        \item [(1)] The centralizer of $H\stackrel{\mathrm{def}}{=}\Aut(k) \subset \Aut_{\Q_{p_{k}}}(k_{+})=\GL_{d_{k}}(\Q_{p_{k}})$ in $\GL_{d_{k}}(\Q_{p_{k}})$ is a closed subgroup of $\GL_{d_{k}}(\Q_{p_{k}})$.
    In particular, we regard the centralizer of $H$ in $\GL_{d_{k}}(\Q_{p_{k}})$ as a $p_{k}$-adic Lie subgroup of $\GL_{d_{k}}(\Q_{p_{k}})$ [cf.\ \cite{17}, Theorem 9.6].
    \item [(2)] When $(\Q_{p_{k}}[H])^{\times}$ is equipped with its natural $p_{k}$-adic Lie group structure, there exists an isomorphism of $p_{k}$-adic Lie groups
    \begin{align*}
        Z_{\GL_{d_{k}}(\Q_{p_{k}})}(H)\stackrel{\sim}{\longrightarrow} (\Q_{p_{k}}[H])^{\times}.
    \end{align*} 
    \item [(3)] The dimension of the $p_{k}$-adic Lie group $Z_{\GL_{d_{k}}(\Q_{p_{k}})}(H)$ is at most $d_{k}$.
    \end{enumerate}
\end{lem}

\begin{proof}
    Since $H=\Aut(k)$ is a finite group, assertion (1) holds.
    Next, we verify assertion (2).
    Since $k$ is absolutely Galois and $H=\Aut(k)=\Gal(k/k^{(d=1)})$, it follows from the normal basis theorem that $k_{+} \cong \Q_{p_{k}}[H]$ as a left $\Q_{p_{k}}[H]$-module [where $k_{+}$ is endowed with its natural left $\Q_{p_{k}}[H]$-module structure].
    By taking the automorphism groups of both sides of the $\Q_{p_{k}}$-isomorphism $k_{+} \cong \Q_{p_{k}}[H]$, we obtain a topological group isomorphism
    \begin{align*}
        Z_{\GL_{d_{k}}(\Q_{p_{k}})}(H) \stackrel{\sim}{\longrightarrow} ((\Q_{p_{k}}[H])^{\times})^{\mathrm{op}} \stackrel{\sim}{\longrightarrow} (\Q_{p_{k}}[H])^{\times},
    \end{align*}
    where we note that the centralizer $Z_{\GL_{d_{k}}(\Q_{p_{k}})}(H)$ is naturally isomorphic to the group of automorphisms of $k_{+}$ as a left $\Q_{p_{k}}[H]$-module.
    By construction, this isomorphism is compatible with the natural $p_{k}$-adic Lie group structure and thus gives an isomorphism of $p_{k}$-adic Lie groups.
    This completes the proof of assertion (2).
    Finally, since the dimension of $(\Q_{p_{k}}[H])^{\times}$ as a $p_{k}$-adic Lie group is at most  $d_{k}$, it follows from assertion (2) that assertion (3) holds.
    This completes the proof of Lemma \ref{the centralizer of Aut(k) in GL}.
\end{proof}

\begin{lem}\label{Sp is not virtually abelian}
    Let $n$ be a positive integer.
    Then the group $\mathrm{Sp}_{2n}(\Z_{p_{k}})$ is not a virtually abelian profinite group, i.e., has no abelian open subgroup.
\end{lem}

\begin{proof}
    Let $U$ be an open subgroup of $\mathrm{Sp}_{2n}(\Z_{p_k})$.
Since $\mathrm{Sp}_{2n}(\Z_{p_k})$ is a compact open subgroup of the $p_k$-adic Lie group $\mathrm{Sp}_{2n}(\Q_{p_k})$, it naturally carries a structure of $p_k$-adic Lie group and, in particular, is profinite.
In general, for a $p_k$-adic Lie group $X$, the Lie algebra structure defined on the tangent space at the identity element of $X$ is naturally identified with the Lie algebra structure defined on the tangent space at the identity element of any open subgroup of $X$.
Thus, the Lie algebra associated with $\mathrm{Sp}_{2n}(\Z_{p_k})$ is naturally identified with the Lie algebra associated with $U$.
We prove Lemma \ref{Sp is not virtually abelian} by contradiction.
Suppose that $U$ is abelian.
Then the Lie bracket on the corresponding Lie algebra is trivial.
However, this contradicts the fact that the Lie bracket on the Lie algebra $\mathfrak{sp}_{2n}(\Q_{p_k})$, which is associated with $\mathrm{Sp}_{2n}(\Z_{p_k})$, is nontrivial.
This completes the proof of Lemma \ref{Sp is not virtually abelian}.
\end{proof}

Suppose that $d_{k} \geq 3$. 
Let $g$ be defined as follows:  
\[
g \stackrel{\mathrm{def}}{=} 
\begin{cases}
\frac{d_k - 1}{2}, & \text{if } d_k \text{ is odd}, \\
\frac{d_k - 2}{2}, & \text{if } d_k \text{ is even}.
\end{cases}
\] 

Let $S$ be a closed orientable surface of genus $g (\geq 1)$.
Furthermore, let $P$ be a point on $S$.
Then it is well-known that there exists an isomorphism
    \begin{align*}
        \pi_{1}(S \setminus\{P\}) \stackrel{\sim}{\longrightarrow} \langle a_{1},b_{1},\ldots,a_{g},b_{g},c\ |\ [a_{1},b_{1}]\cdots[a_{g},b_{g}]c=1 \rangle,
    \end{align*}
    which is induced by standard generating loops of $S \setminus {P}$ [cf., e.g., \cite{18}, the discussion following Proposition 1.26 and related arguments].
    Define $\delta\stackrel{\mathrm{def}}{=}[a_{1},b_{1}] \cdots [a_{g},b_{g}]=c^{-1}$ and $\mathrm{Aut}^{\delta}(\pi_1(S \setminus \{P\})) \subset \Aut(\pi_{1}(S \setminus \{P\}))$ to be the subgroup of automorphisms of $\pi_{1}(S \setminus \{P\})$ that fix $\delta \in \pi_{1}(S \setminus \{P\})$.
    For each integer $i$ satisfying $1 \leq i \leq g$, 
we assign
\begin{align*}
    a_{i} \mapsto \left\{
\begin{aligned}
&x_{2i+1}\ (\text{if $d_{k}$ is even})\\
&x_{2i}\  (\text{if $d_{k}$ is odd}),
\end{aligned}
\right.\ 
b_{i} \mapsto \left\{
\begin{aligned}
&x_{2i+2}\ (\text{if $d_{k}$ is even})\\
&x_{2i+1}\  (\text{if $d_{k}$ is odd}).
\end{aligned}
\right.\ 
\end{align*}
    It follows from Theorem \ref{generator and relation of group of MLF-type} that, under this assignment, automorphism $\varphi$ of $\pi_{1}(S \setminus \{P\})$ that fixes $\delta$ induces an automorphism of $G_{k}$, which, by abuse of notation, we denote by $\varphi$, as follows:
    As for the generators other than the $x_i$ corresponding to $a_i$ and $b_i$, we define $\varphi$ according to the parity of $d_k$ as follows:
\begin{itemize}
    \item If $d_k$ is odd, then
    \[
        \varphi(\sigma) = \sigma,\quad \varphi(\tau) = \tau,\quad \varphi(x_0) = x_0,\quad \varphi(x_1) = x_1.
    \]
    \item If $d_k$ is even, then
    \[
        \varphi(\sigma) = \sigma,\quad \varphi(\tau) = \tau,\quad \varphi(x_0) = x_0,\quad \varphi(x_1) = x_1,\quad \varphi(x_2) = x_2.
    \]
\end{itemize}

    Let $\Mod(S \setminus\{P\})$ be the mapping class group of $S \setminus\{P\}$ [see \cite{16}, the beginning of Chapter 2].
    Consider the natural homomorphism $\Mod(S \setminus\{P\}) \to \Out(\pi_{1}(S \setminus \{P\}))$.  
It is well-known that the natural map $\Mod(S \setminus \{P\}) \to \Out(\pi_{1}(S \setminus \{P\}))$ is injective, and that there exists a lifting in $\Aut(\pi_{1}(S \setminus \{P\}))$ of the image of each element of $\Mod(S \setminus \{P\})$ which fixes $\delta$ [cf.\ \cite{16}, Theorem 4.1, and the remark following \cite{16}, Theorem 4.1].
Throughout the remainder of the present paper, we regard $\Mod(S \setminus \{P\})$ as a subgroup of $\Out(\pi_{1}(S \setminus \{P\}))$ via the natural injective group homomorphism $\Mod(S \setminus \{P\}) \hookrightarrow \Out(\pi_{1}(S \setminus \{P\}))$.
Let $\varphi$ be an element of $\Mod(S \setminus \{P\})$.
As mentioned above, there exists a lifting of $\varphi$ to $\Aut(\pi_{1}(S \setminus \{P\}))$ which fixes $\delta$.
Hence, each such lifting induces an automorphism of $G_{k}$.

 We obtain a map of sets
\begin{align*}
   \rho \colon \Mod(S \setminus \{P\}) \to \Out(G_k),
\end{align*}
which is not necessarily a homomorphism of groups, by considering the diagram of sets

\[
\begin{tikzcd}[row sep=large, column sep=large]
& \mathrm{Aut}^{\delta}(\pi_1(S \setminus \{P\})) \arrow[r, hook] \arrow[d, hook] & \mathrm{Aut}(G_k) \arrow[d] \\
& \mathrm{Aut}(\pi_1(S \setminus \{P\})) \arrow[d] & \mathrm{Out}(G_k) \\
\Mod(S \setminus \{P\}) \arrow[r, hook] & \mathrm{Out}(\pi_1(S \setminus \{P\})). &
\end{tikzcd}
\]
Write $D \subset \Out(G_k)$ for the closed subgroup of $\Out(G_k)$ that is topologically generated by the image of $\rho$.

\begin{rem}\label{mapping class group meaning}
   Let us consider the automorphisms $\varphi_{i}$, $\varphi^{\prime}_{i}$, and $\varphi^{\prime\prime}_{i}$ of $G_{k}$, which are defined in the proof of Theorem \ref{determination of Ker(Tr)}.  
   Then we observe that, in light of the relation of $\Out(G_{k})$ with $\Mod(S \setminus \{P\})$ discussed above, for each integer $1 \leq i \leq g$, $\varphi_{i}$ corresponds to the Dehn twist determined by a longitude, $\varphi^{\prime}_{i}$ corresponds to the Dehn twist determined by a meridian, and, if $g \geq 2$, then the element $\varphi^{\prime\prime}_{i}$ corresponds to the Dehn twist determined by a loop spanning two adjacent holes of a torus [cf.\ \cite{14}, Lemma 7.3.4].
\end{rem}

To prove the main theorem, we introduce a classical result in topology.

\begin{thm}[\cite{16}]\label{The Symplectic Representation}
      Recall that the images $[a_{1}], [b_{1}],\ldots, [a_{g}], [b_{g}] \in \pi_{1}(S \setminus \{P\})^{\ab}$ of $a_{1},b_{1}, \ldots, a_{g},b_{g} \in \pi_{1}(S \setminus \{P\})$ form a basis of the free $\Z$-module $\pi_{1}(S \setminus \{P\})^{\ab}$ [cf.\ the canonical isomorphism $\pi_{1}(S \setminus \{P\})^{\ab} \stackrel{\sim}{\longrightarrow} H_{1}(S\setminus \{P\},\Z)$].
     Consider the representation
\begin{align*}
    \Mod(S \setminus\{P\}) \to \Out(\pi_{1}(S \setminus \{P\})) \to \Aut(\pi_{1}(S \setminus \{P\})^{\ab})\stackrel{\sim}{\longrightarrow} \GL_{2g}(\Z),
\end{align*}
where the third arrow is an isomorphism induced by the basis of $\pi_{1}(S \setminus \{P\})^{\ab}$ mentioned above.
Then the image of this map is $\mathrm{Sp}_{2g}(\Z)$.
\end{thm}

\begin{proof}
    This assertion follows from the naturality of the isomorphism $\pi_{1}(S \setminus \{P\})^{\ab} \stackrel{\sim}{\longrightarrow} H_{1}(S \setminus \{P\},\Z)$; \cite{16}, Theorem 6.4; and the remark following \cite{16}, Theorem 6.4.
\end{proof}

We now prove the non-normality of field-theoretic subgroups in the case where $d_{k}>1$ is odd.
This result [cf.\ Theorem \ref{non-normality when d_{k} is odd} below] combined with Theorem \ref{non-normality when d_{k} is even} gives Theorem I [cf. Introduction of the present paper].
In the remainder of the present paper, unless otherwise specified, we assume that $d_{k}>1$ is odd.

In the following proof, we regard $\mathrm{Sp}_{2g}(\Z_{p_{k}})$ as a subgroup of $\GL_{d_{k}}(\Q_{p_{k}})$ via the injective group homomorphism which is defined by
\[
    A \mapsto \begin{bmatrix}
    1 & 0_{1 \times 2g} \\
    0_{2g \times 1} & A
    \end{bmatrix},
\]
where $0_{1 \times 2g}$ [respectively, $0_{2g \times 1}$] denotes the $1 \times 2g$ matrix [respectively, the $2g \times 1$ matrix] whose entries are all $0$.

\begin{rem}\label{Def of Phi}
    We write $\Phi$ for the group homomorphism $\Out(G_k) \to \Aut_{\Q_{p_{k}}}(k_{+}) \stackrel{\sim}{\longrightarrow} \GL_{d_k}(\Q_{p_k})$ induced by the mono-anabelian reconstruction algorithm.
    Here, we note that it is not clear to the author at the time of writing whether or not it is possible to choose suitable liftings so that $\rho$ becomes a group homomorphism.
    On the other hand, it follows from the definitions of the mono-anabelian reconstruction algorithm and of $\rho$ that the following diagram commutes:
    \[
\begin{tikzcd}
\Out(G_{k}) \arrow[r, "\Phi"] & \GL_{d_{k}}(\Q_{p_{k}}) \\
\Mod(S \setminus\{P\}) \arrow[r,] \arrow[u, "\rho"] & \mathrm{Sp}_{2g}(\Z) \arrow[u, "\subset"],
\end{tikzcd}
\]
where the lower horizontal arrow is the surjective homomorphism discussed in Theorem \ref{The Symplectic Representation}.
What is crucial in the proof of Theorem \ref{non-normality when d_{k} is odd} below is the commutativity of this diagram and the surjectivity of $\Mod(S \setminus {P}) \to \mathrm{Sp}_{2g}(\Z)$.
\end{rem}

In what follows, for a $p_{k}$-adic Lie group $X$, we write $\dim(X)$ for the dimension of $X$ as a $p_{k}$-adic Lie group.

\begin{thm}\label{non-normality when d_{k} is odd}
     Suppose that $k$ is an absolutely Galois MLF, that $p_{k}$ is odd, and that $d_{k}>1$ is odd.
    Then the set of $\Out(G_{k})$-conjugates of the subgroup $\Aut(k) \subset \Out(G_{k})$ is infinite.
    In particular, $\Aut(k) \subset \Out(G_{k})$ is not a normal subgroup of $\Out(G_{k})$.
\end{thm}

\begin{proof}
    It follows from the existence of the natural exact sequence
    \begin{align*}
        1 \to Z_{\Out(G_{k})}(\Aut(k)) \to N_{\Out(G_{k})}(\Aut(k)) \to \Aut(\Aut(k)),
    \end{align*}
    together with the fact that $\Aut(k)$ is a finite group, that the subgroup $Z_{\Out(G_{k})}(\Aut(k)) \subset N_{\Out(G_{k})}(\Aut(k))$ is of finite index.
    Thus, in order to show that the subgroup
    \begin{align*}
        N_{\Out(G_k)}(\Aut(k)) \subset \Out(G_k)
    \end{align*}
    is not of finite index, it suffices to show that the subgroup $Z_{\Out(G_k)}(\Aut(k)) \subset \Out(G_k)$ is not of finite index.
    
    We first consider the case where $g \geq 2$.
It follows from Lemma \ref{closed map} that the map $\Phi$ [cf. Remark \ref{Def of Phi}] is a closed map.
Thus, both $I \stackrel{\mathrm{def}}{=} \Phi(\Out(G_k))$ and $Z \stackrel{\mathrm{def}}{=} \Phi(Z_{\Out(G_k)}(\Aut(k)))$ are closed subgroups of $\GL_{d_k}(\Q_{p_k})$.
In particular, it follows from \cite{17}, Theorem 9.6, that both $I$ and $Z$ naturally admit structures of $p_k$-adic Lie groups.
Furthermore, it follows from Remark \ref{Def of Phi} that the image of $D$ [cf. the discussion preceding Remark \ref{mapping class group meaning}] by $\Phi$ is equal to $\mathrm{Sp}_{2g}(\Z_{p_k})$, which implies that $\mathrm{Sp}_{2g}(\Z_{p_k}) \subset I$.
As is well-known, since the dimension of $\mathrm{Sp}_{2g}(\Z_{p_k})$ as a $p_k$-adic Lie group is $2g^{2} + g$, it follows that $\dim(I) \geq 2g^2 + g$.

Let us now estimate the dimension of $Z$ as a $p_k$-adic Lie group.
By definition, we have $Z \subset Z_{\GL_{d_{k}}(\Q_{p_{k}})}(\Aut(k))$.
Thus, it follows from \ref{the centralizer of Aut(k) in GL}, (3), that the  dimension of $Z_{\GL_{d_{k}}(\Q_{p_{k}})}(\Aut(k))$ as a $p_k$-adic Lie group is at most $2g + 1$.
Hence, we obtain $\dim(Z) \leq 2g + 1$.
Since $g \geq 2$, we have $2g^2 + g > 2g + 1$.
Therefore, the dimension comparison shows that $Z$ is not an open subgroup of $I$.
Indeed, if $Z$ were open in $I$, then $\dim(Z) = \dim(I)$ would hold, which contradicts the inequality $2g^2 + g > 2g + 1$.
On the other hand, since $Z$ is a closed subgroup of $I$ [cf.\ Lemma \ref{closed map}], it follows that $Z$ is not of finite index in $I$.
In particular, we conclude that the subgroup $N_{\Out(G_k)}(\Aut(k)) \subset \Out(G_k)$ is not of finite index.

Finally, we consider the case where $g = 1$, i.e., $d_{k} = 3$.
Suppose, for the sake of contradiction, that $Z_{\Out(G_k)}(\Aut(k)) \subset \Out(G_k)$ is a subgroup of finite index.
Then since the group $\Aut(k) = \Gal(k/k^{(d=1)})$ is abelian and $Z_{\GL_{d_{k}}(\Q_{p_{k}})}(\Aut(k)) \cong (\Q_{p_{k}}[\Aut(k)])^{\times}$ [cf.\ Lemma \ref{the centralizer of Aut(k) in GL}, (2)], it follows that $Z_{\GL_{d_{k}}(\Q_{p_{k}})}(\Aut(k))$ is also abelian.
Thus, the profinite group $I$ is virtually abelian, i.e., has an abelian open subgroup.
Since any closed subgroup of a virtually abelian profinite group is a virtually abelian profinite group, it follows that the closed subgroup $\mathrm{Sp}_{2g}(\Z_{p_{k}}) \subset I$ is a virtually abelian profinite group.
However, this contradicts Lemma \ref{Sp is not virtually abelian}.
This completes the proof of Theorem \ref{non-normality when d_{k} is odd}.
\end{proof}

\begin{rem}\label{summary of main result}
  Suppose that $k$ is absolutely Galois, that $p_{k}$ is odd, and that $d_{k} \geq 2$.
  Then it follows from Theorem \ref{non-normality when d_{k} is even} and Theorem \ref{non-normality when d_{k} is odd} that the subgroup $\Aut (k) \subset \Out(G_{k})$ is not normal.
  Here, we note that this non-normality implies that the subgroup $\Aut (k) \subset \Out(G_{k})$ cannot be reconstructed functorially and group-theoretically from $G_{k}$.
  Indeed, suppose to the contrary that such a mono-anabelian reconstruction algorithm exists.
  Then it would imply that, for any automorphism $\alpha \colon G_{k} \stackrel{\sim}{\longrightarrow} G_{k}$ of the group $G_{k}$, we have $\Out(\alpha)(\Aut(k))=\Aut(k)$, where we write $\Out(\alpha) \colon \Out(G_{k}) \stackrel{\sim}{\longrightarrow} \Out(G_{k})$ for the automorphism of the group $\Out(G_{k})$ induced from $\alpha$.
  On the other hand, it follows formally that the automorphism $\Out(\alpha)$ coincides with the inner automorphism determined by the image of $\alpha$ by the natural surjective group homomorphism $\Aut(G_{k}) \to \Out(G_{k})$.
  This implies that $\Aut(k) \subset \Out(G_{k})$ is normal, which contradicts our non-normality result.
\end{rem}

\begin{rem}\label{odd-case}
Observe that one verifies immediately that the method applied in the proof of Theorem \ref{non-normality when d_{k} is odd} [i.e., in the odd degree case] can also be applied to prove Theorem \ref{non-normality when d_{k} is even} [i.e., in the even degree case].
We leave the routine details to the interested reader.
On the other hand, in contrast to the proof of Theorem \ref{non-normality when d_{k} is even}, the proof of Theorem \ref{non-normality when d_{k} is odd} does not enable us to explicitly construct elements of $\Out(G_k) \setminus N_{\Out(G_k)}(\Aut(k))$.
That is to say, although its scope is more limited, the method applied in the proof of Theorem \ref{non-normality when d_{k} is even} has the advantage of providing explicit elements of the desired type.
Indeed, Theorem \ref{Aut-intristic Hodge-Tate} was obtained through such an observation.
\end{rem}

\begin{rem}
    When $p_{k}=2$ and $\sqrt{-1} \in k$, as in the case where $p_{k}$ is odd, an explicit description by generators and relations for the topological group $G_{k}$ [as discussed in Theorem \ref{generator and relation of group of MLF-type}] is known.
    In fact, in \cite{21}, Nishio proved similar results to Theorem \ref{Aut-intristic Hodge-Tate} and Theorem \ref{non-normality when d_{k} is even} in the situation where $k$ is absolutely abelian, and $\sqrt{-1} \in k$ by means of the explicit description.
    On the other hand, when $p_{k}=2$ and $\sqrt{-1} \notin k$, a similar explicit description by generators and relations for the topological group $G_{k}$ is not known.
    In particular, it is not clear to the author at the time of writing whether or not the subgroup $\Aut(k) \subset \Out(G_{k})$ is a normal subgroup in the case where $p_{k}=2$ and $\sqrt{-1} \notin k$, even when $k$ is absolutely Galois.
\end{rem}

\begin{rem}
    If $d_{k} = 1$, then it is easy to see that $\Aut(k) = 1$.
    In particular, the subgroup $\Aut(k) \subset \Out(G_{k})$ is a normal subgroup.
    On the other hand, when $k$ is not absolutely Galois and $\Aut(k) \neq 1$, the non-normality of the field-theoretic subgroup cannot be established by the method developed in the present paper.
   In a joint work in preparation by the author and Tsujimura, we are studying similar results to the various results obtained in the present paper under the assumption that $\Aut(k)\neq 1$ without assuming that $k$ is absolutely Galois.
    Moreover, we are also studying analogous results for positive-characteristic local fields and, furthermore, for certain special henselian valued fields that are not necessarily local fields.
\end{rem}

\begin{rem}
    Let $\alpha \in \Out(G_{k})$.
    As discussed in Remark \ref{ring theoretic deta}, if  $ \alpha \in N_{\Out(G_{k})}(\Aut(k))$, then it preserves certain objects that essentially depend on the ring structure of $k$.
    On the other hand, by Theorem \ref{non-normality when d_{k} is even} and Theorem \ref{non-normality when d_{k} is odd}, if $k$ is absolutely Galois and $d_{k}>1$, then the subgroup $N_{\Out(G_{k})}(\Aut(k)) \subset \Out(G_{k})$ is a proper subgroup.
    In light of these observations, it is natural to ask whether or not the elements of the subgroup $N_{\Out(G_{k})}(\Aut(k)) \subset \Out(G_{k})$ can be characterized, as in Theorem \ref{characterization of Aut k} [cf.\ Introduction of the present paper], as those that preserve suitable ``ring-theoretic data''.
   However, at the time of writing the present paper, we have not yet obtained even a candidate for an appropriate characterization.
\end{rem}

\section*{Acknowledgments}

I would like to express my sincere gratitude to Professor Yuichiro Hoshi, Professor Shinichi Mochizuki, Professor Shota Tsujimura, and Professor Akio Tamagawa for their numerous insightful discussions and warm encouragement.
I am also grateful to Reiya Tachihara for carefully reading my drafts and providing invaluable comments.
I am especially thankful to Professor Yuichiro Hoshi for suggesting the theme of the present paper and for his detailed feedback on the drafts, and to Professor Akio Tamagawa for his remarks on a critical idea in the proof of Theorem \ref{non-normality when d_{k} is odd}.
Finally, I would like to express my heartfelt appreciation to my family, significant other, and friends for their unwavering support.

\end{document}